\documentclass{article}

\usepackage{graphicx}
\usepackage{amsmath,amsthm,amssymb,amsfonts}
\usepackage{verbatim}
\usepackage{stmaryrd} 
\usepackage{hyperref}

\usepackage[colorinlistoftodos]{todonotes}

\usepackage{tikz}
\usepackage{tkz-berge, tkz-graph}
\tikzset{vertex/.style={circle,draw,fill,inner sep=0pt,minimum size=1mm}}

\theoremstyle{plain}
\newtheorem{thm}{Theorem}
\newtheorem{lem}[thm]{Lemma}
\newtheorem{prop}[thm]{Proposition}
\newtheorem{cor}[thm]{Corollary}

\theoremstyle{definition}
\newtheorem{definition}[thm]{Definition}
\newtheorem{exl}[thm]{Example}

\numberwithin{thm}{section}

\def\N{{\mathbb N}}
\def\Z{{\mathbb Z}}
\def\R{{\mathbb R}}

\begin{document}
\title{Generalized Normal Product Adjacency in Digital Topology}
\author{Laurence Boxer
         \thanks{
    Department of Computer and Information Sciences,
    Niagara University,
    Niagara University, NY 14109, USA;
    and Department of Computer Science and Engineering,
    State University of New York at Buffalo.
    E-mail: boxer@niagara.edu
    }
}

\date{ }
\maketitle

\begin{abstract}
We study properties of Cartesian products of digital images for which
adjacencies based on the
normal product adjacency are used. We
show that the use of such adjacencies
lets us obtain many ``product properties'' for which the analogous
statement is either unknown or invalid if, instead, we
were to use $c_u$-adjacencies.

Key words and phrases: digital topology, digital image, continuous multivalued function, shy map, retraction
\end{abstract}

\section{Introduction}
We study adjacency relations based on
the normal product adjacency for Cartesian products of
multiple digital images. Most of the literature
of digital topology focuses on images that use a
$c_u$-adjacency; however, the results of this paper seem to indicate that for Cartesian products of
digital images, the natural adjacencies to use are
based on the normal product adjacency of the factor
adjacencies, in the sense of preservation of many 
properties in Cartesian products.

\section{Preliminaries}
\label{prelims}
We use $\N$, $\Z$, and $\R$ to represent
the sets of natural numbers, integers, and 
real numbers, respectively,  

Much of the material that appears in this section is quoted or paraphrased
from~\cite{BoxSta16,BoSt1}, and other papers cited in this section.

We will assume familiarity with the topological theory of digital images. See, e.g., \cite{Boxer94} for many of the standard definitions. All digital images $X$ are assumed to carry their own adjacency relations (which may differ from one image to another). When we wish to emphasize the particular adjacency relation we write the image as $(X,\kappa)$, where $\kappa$ represents
the adjacency relation.

\subsection{Common adjacencies}
Among the commonly used adjacencies are the $c_u$-adjacencies.
Let $x,y \in \Z^n$, $x \neq y$. Let $u$ be an integer,
$1 \leq u \leq n$. We say $x$ and $y$ are $c_u$-adjacent if
\begin{itemize}
\item There are at most $u$ indices $i$ for which 
      $|x_i - y_i| = 1$.
\item For all indices $j$ such that $|x_j - y_j| \neq 1$ we
      have $x_j=y_j$.
\end{itemize}
We often label a $c_u$-adjacency by the number of points
adjacent to a given point in $\Z^n$ using this adjacency.
E.g.,
\begin{itemize}
\item In $\Z^1$, $c_1$-adjacency is 2-adjacency.
\item In $\Z^2$, $c_1$-adjacency is 4-adjacency and
      $c_2$-adjacency is 8-adjacency.
\item In $\Z^3$, $c_1$-adjacency is 6-adjacency,
      $c_2$-adjacency is 18-adjacency, and $c_3$-adjacency
      is 26-adjacency.
\end{itemize}

Given digital images or graphs $(X,\kappa)$ and $(Y,\lambda)$, the
{\em normal product adjacency} $NP(\kappa,\lambda)$ (also called the
{\em strong adjacency}~
\cite{vLW} and denoted
$\kappa_*(\kappa,\lambda)$ in~\cite{BoxKar12}) generated by
$\kappa$ and $\lambda$ on the Cartesian product $X \times Y$ is defined
as follows.

\begin{definition}
\label{NP-def}
\rm{\cite{Berge}}
Let $x, x' \in X$, $y, y' \in Y$.
Then $(x,y)$ and $(x',y')$ are $NP(\kappa,\lambda)$-adjacent in $X \times Y$
if and only if
\begin{itemize}
\item $x=x'$ and $y$ and $y'$ are $\lambda$-adjacent; or
\item $x$ and $x'$ are $\kappa$-adjacent and $y=y'$; or
\item $x$ and $x'$ are $\kappa$-adjacent and $y$ and $y'$ are $\lambda$-adjacent. \qed
\end{itemize}
\end{definition}

\subsection{Connectedness}
A subset $Y$ of a digital image $(X,\kappa)$ is
{\em $\kappa$-connected}~\cite{Rosenfeld},
or {\em connected} when $\kappa$
is understood, if for every pair of points $a,b \in Y$ there
exists a sequence $\{y_i\}_{i=0}^m \subset Y$ such that
$a=y_0$, $b=y_m$, and $y_i$ and $y_{i+1}$ are 
$\kappa$-adjacent for $0 \leq i < m$.

For two subsets $A,B\subset X$, we will say that $A$ and $B$ are \emph{adjacent} when there exist points $a\in A$ and $b\in B$ such that $a$ and $b$ are equal or adjacent. Thus sets with nonempty intersection are automatically adjacent, while disjoint sets may or may not be adjacent. It is easy to see that a finite union of connected adjacent sets is connected.


\subsection{Continuous functions}
The following generalizes a definition of
~\cite{Rosenfeld}.

\begin{definition}\label{continuous}
{\rm ~\cite{Boxer99}}
Let $(X,\kappa)$ and $(Y,\lambda)$ be digital images. A function
$f: X \rightarrow Y$ is $(\kappa,\lambda)$-continuous if for
every $\kappa$-connected $A \subset X$ we have that
$f(A)$ is a $\lambda$-connected subset of $Y$. 
\end{definition}

When the adjacency relations are understood, we will simply say that $f$ is \emph{continuous}. Continuity can be reformulated in terms of adjacency of points:
\begin{thm}
\label{cont-by-adj}
{\rm ~\cite{Rosenfeld,Boxer99}}
A function $f:X\to Y$ is continuous if and only if, for any adjacent points $x,x'\in X$, the points $f(x)$ and $f(x')$ are equal or adjacent. \qed
\end{thm}

Note that similar notions appear
in~\cite{Chen94,Chen04} under the names
{\em immersion}, {\em gradually varied operator},
and {\em gradually varied mapping}.

\begin{thm}
\label{composition}
\rm{\cite{Boxer94,Boxer99}}
If $f: (A,\kappa) \to (B,\lambda)$ and $g: (B,\lambda) \to (C, \mu)$ are
continuous, then $g \circ f: (A,\kappa) \to (C, \mu)$ is continuous. \qed
\end{thm}

\begin{exl}
\rm{\cite{Rosenfeld}}
\label{const-exl}
A constant function between digital images is continuous. \qed
\end{exl}

\begin{exl}
\label{id-exl}
The identity function $1_X: (X, \kappa) \to (X, \kappa)$ is continuous.
\end{exl}

\begin{proof} This is an immediate consequence of Theorem~\ref{cont-by-adj}.
\end{proof}

\begin{definition}
Let $(X,\kappa)$ be a digital image
in $\Z^n$.
Let $x,y \in X$. A {\em 
$\kappa$-path of length $m$ from $x$ to $y$}
is a set $\{x_i\}_{i=0}^m \subset X$
such that $x=x_0$, $x_m=y$, and
$x_{i-1}$ and $x_i$ are equal or $\kappa$-adjacent for $1 \leq i \leq m$. If $x=y$, we say $\{x\}$ is a
{\em path of length 0 from $x$ to $x$}.
\end{definition}

Notice that for a path from $x$ to $y$ as
described above, the function
$f: [0,m]_{\Z} \to X$ defined by
$f(i)=x_i$ is $(c_1,\kappa)$-continuous. Such
a function is also called a {\em 
$\kappa$-path of length $m$ from $x$ to $y$}.

\subsection{Digital homotopy}
A homotopy between continuous functions may be thought of as
a continuous deformation of one of the functions into the 
other over a finite time period.

\begin{definition}{\rm (\cite{Boxer99}; see also \cite{Khalimsky})}
\label{htpy-2nd-def}
Let $X$ and $Y$ be digital images.
Let $f,g: X \rightarrow Y$ be $(\kappa,\kappa')$-continuous functions.
Suppose there is a positive integer $m$ and a function
$F: X \times [0,m]_{{\Z}} \rightarrow Y$
such that

\begin{itemize}
\item for all $x \in X$, $F(x,0) = f(x)$ and $F(x,m) = g(x)$;
\item for all $x \in X$, the induced function
      $F_x: [0,m]_{{\Z}} \rightarrow Y$ defined by
          \[ F_x(t) ~=~ F(x,t) \mbox{ for all } t \in [0,m]_{{\Z}} \]
          is $(2,\kappa')-$continuous. That is, $F_x(t)$ is a path in $Y$.
\item for all $t \in [0,m]_{{\Z}}$, the induced function
         $F_t: X \rightarrow Y$ defined by
          \[ F_t(x) ~=~ F(x,t) \mbox{ for all } x \in  X \]
          is $(\kappa,\kappa')-$continuous.
\end{itemize}
Then $F$ is a {\rm digital $(\kappa,\kappa')-$homotopy between} $f$ and
$g$, and $f$ and $g$ are {\rm digitally $(\kappa,\kappa')-$homotopic in} $Y$.
If for some $x \in X$ we have $F(x,t)=F(x,0)$ for all
$t \in [0,m]_{{\Z}}$, we say $F$ {\rm holds $x$ fixed}, and $F$ is a {\rm pointed homotopy}.
$\Box$
\end{definition}

We denote a pair of homotopic functions as
described above by $f \simeq_{\kappa,\kappa'} g$.
When the adjacency relations $\kappa$ and $\kappa'$ are understood in context,
we say $f$ and $g$ are {\em digitally homotopic} (or just {\em homotopic})
to abbreviate ``digitally 
$(\kappa,\kappa')-$homotopic in $Y$," and write
$f \simeq g$.

\begin{prop}
\label{htpy-equiv-rel}
{\rm ~\cite{Khalimsky,Boxer99}}
Digital homotopy is an equivalence relation among
digitally continuous functions $f: X \rightarrow Y$.
$\Box$
\end{prop}



\begin{definition}
{\rm ~\cite{Boxer05}}
\label{htpy-type}
Let $f: X \rightarrow Y$ be a $(\kappa,\kappa')$-continuous function and let
$g: Y \rightarrow X$ be a $(\kappa',\kappa)$-continuous function such that
\[ f \circ g \simeq_{\kappa',\kappa'} 1_X \mbox{ and }
   g \circ f \simeq_{\kappa,\kappa} 1_Y. \]
Then we say $X$ and $Y$ have the {\rm same $(\kappa,\kappa')$-homotopy type}
and that $X$ and $Y$ are $(\kappa,\kappa')$-{\rm homotopy equivalent}, denoted 
$X \simeq_{\kappa,\kappa'} Y$ or as
$X \simeq Y$ when $\kappa$ and $\kappa'$ are
understood.
If for some $x_0 \in X$ and $y_0 \in Y$ we have
$f(x_0)=y_0$, $g(y_0)=x_0$,
and there exists a homotopy between $f \circ g$
and $1_X$ that holds $x_0$ fixed, and 
a homotopy between $g \circ f$
and $1_Y$ that holds $y_0$ fixed, we say
$(X,x_0,\kappa)$ and $(Y,y_0,\kappa')$ are
{\rm pointed homotopy equivalent} and that $(X,x_0)$ 
and $(Y,y_0)$ have the 
{\rm same pointed homotopy type}, denoted 
$(X,x_0) \simeq_{\kappa,\kappa'} (Y,y_0)$ or as
$(X,x_0) \simeq (Y,y_0)$ when 
$\kappa$ and $\kappa'$ are understood.
$\Box$
\end{definition}

It is easily seen, from 
Proposition~\ref{htpy-equiv-rel}, that having the
same homotopy type (respectively, the same
pointed homotopy type) is an equivalence relation
among digital images (respectively, among pointed
digital images).

\subsection{Continuous and connectivity preserving multivalued functions}
A \emph{multivalued function} $f:X\to Y$ assigns a subset of $Y$ to each point of $x$. We will  write $f:X \multimap Y$. For $A \subset X$ and a multivalued function $f:X\multimap Y$, let $f(A) = \bigcup_{x \in A} f(x)$. 

\begin{definition}
\label{mildly}
\rm{\cite{Kovalevsky}}
A multivalued function $f:X\multimap Y$ is \emph{connectivity preserving} if $f(A)\subset Y$ is connected whenever $A\subset X$ is connected.
\end{definition}

As is the case with Definition \ref{continuous}, we can reformulate connectivity preservation in terms of adjacencies.

\begin{thm}
\rm{\cite{BoxSta16}}
\label{mildadj}
A multivalued function $f:X \multimap Y$ is \emph{connectivity preserving} if and only if the following are satisfied:
\begin{itemize}
\item For every $x \in X$, $f(x)$ is a connected subset of $Y$.
\item For any adjacent points $x,x'\in X$, the sets $f(x)$ and $f(x')$ are adjacent. \qed
\end{itemize}
\end{thm}

Definition~\ref{mildly} is related to a definition of multivalued continuity for subsets of $\Z^n$ given and explored by Escribano, Giraldo, and Sastre in \cite{egs08, egs12} based on subdivisions. (These papers make a small error with respect to compositions, that is corrected in \cite{gs15}.) Their definitions are as follows:
\begin{definition}
For any positive integer $r$, the \emph{$r$-th subdivision} of $\Z^n$ is
\[ \Z_r^n = \{ (z_1/r, \dots, z_n/r) \mid z_i \in \Z \}. \]
An adjacency relation $\kappa$ on $\Z^n$ naturally induces an adjacency relation (which we also call $\kappa$) on $\Z_r^n$ as follows: $(z_1/r, \dots, z_n/r), (z'_1/r, \dots, z'_n/r)$ are adjacent in $\Z^n_r$ if and only if $(z_1, \dots, z_n)$ and $(z_1, \dots, z_n)$ are adjacent in $\Z^n$.

Given a digital image $(X,\kappa) \subset (\Z^n,\kappa)$, the \emph{$r$-th subdivision} of $X$ is 
\[ S(X,r) = \{ (x_1,\dots, x_n) \in \Z^n_r \mid (\lfloor x_1 \rfloor, \dots, \lfloor x_n \rfloor) \in X \}. \]

Let $E_r:S(X,r) \to X$ be the natural map sending $(x_1,\dots,x_n) \in S(X,r)$ to $(\lfloor x_1 \rfloor, \dots, \lfloor x_n \rfloor)$. \qed
\end{definition}

\begin{definition}
For a digital image $(X,\kappa) \subset (\Z^n,\kappa)$, a function $f:S(X,r) \to Y$ \emph{induces a multivalued function $F:X\multimap Y$} if $x \in X$ implies
\[ F(x) = \bigcup_{x' \in E^{-1}_r(x)} \{f(x')\}. \qed \]
\end{definition}

\begin{definition}
\label{multi-cont}
A multivalued function $F:X\multimap Y$ is called \emph{continuous} when there is some $r$ such that $F$ is induced by some single valued continuous function $f:S(X,r) \to Y$. 
\qed
\end{definition}

\begin{figure}
\begin{center}
\includegraphics[height=1.25
in]{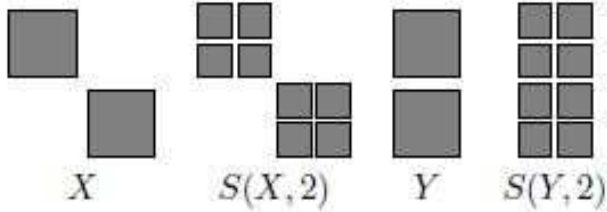}
\end{center}
\caption{\cite{BoxSta16} Two images $X$ and $Y$ with their second subdivisions. 
\label{subdivfig}}
\end{figure}

Note~\cite{BoxSta16} that the subdivision construction (and thus the notion of continuity) depends on the particular embedding of $X$ as a subset of $\Z^n$. In particular we may have $X, Y \subset \Z^n$ with $X$ isomorphic to $Y$ but $S(X,r)$ not isomorphic to $S(Y,r)$. This in fact is the case for the two images in Figure~\ref{subdivfig}, when we use 8-adjacency for all images. Then the spaces $X$ and $Y$ in the figure are isomorphic, each being a set of two adjacent points. But $S(X,2)$ and $S(Y,2)$ are not isomorphic since $S(X,2)$ can be disconnected by removing a single point, while this is impossible in $S(Y,2)$. 

The definition of connectivity preservation makes no reference to $X$ as being embedded inside of any particular integer lattice $\Z^n$.

\begin{prop}
\label{pt-images-connected}
\rm{\cite{egs08,egs12}}
Let $F:X\multimap Y$ be a continuous multivalued function
between digital images. Then
\begin{itemize}
\item for all $x \in X$, $F(x)$ is connected; and
\item for all connected subsets $A$ of $X$, $F(A)$ is connected.
\qed
\end{itemize}
\end{prop}

\begin{thm}
\label{cont-hierarchy}
\rm{\cite{BoxSta16}}
For $(X,\kappa) \subset (\Z^n,\kappa)$, if $F:X\multimap Y$ 
is a continuous multivalued function, then $F$ is connectivity preserving. \qed
\end{thm}

The subdivision machinery often makes it difficult to prove that a given multivalued function is continuous. By contrast, many maps can easily be shown to be connectivity preserving. 

\begin{prop}
\label{1-to-all}
\rm{\cite{BoxSta16}}
Let $X$ and $Y$ be digital images.
Suppose $Y$ is connected. Then the
multivalued function $f: X \multimap Y$ defined by
$f(x)=Y$ for all $x \in X$ is connectivity preserving. \qed
\end{prop}

\begin{prop}
\label{finite-to-infinite}
\rm{\cite{BoxSta16}}
Let $F: (X,\kappa) \multimap (Y,\lambda)$ be a multivalued
surjection between digital images 
$(X,\kappa),(Y,\kappa)\subset (\Z^n, \kappa)$. If $X$ is finite and $Y$
is infinite, then $F$ is not continuous. \qed
\end{prop}

\begin{cor}
\rm{\cite{BoxSta16}}
Let $F: X \multimap Y$ be the
multivalued function
between digital images defined by
$F(x)=Y$ for all $x \in X$. If $X$ is finite and $Y$
is infinite and connected, then
$F$ is connectivity preserving but not continuous. \qed
\end{cor}

Examples of connectivity preserving but not continuous multivalued functions on finite spaces are given in \cite{BoxSta16}.

\subsection{Other notions of multivalued continuity}
Other notions of continuity have been given
for multivalued functions between graphs (equivalently,
between digital images). We have the following.

\begin{definition}
\rm{~\cite{Tsaur}}
\label{Tsaur-def}
Let $F: X \multimap Y$ be a multivalued function between
digital images.
\begin{itemize}
\item $F$ has {\em weak continuity} if for each pair of
      adjacent $x,y \in X$, $f(x)$ and $f(y)$ are adjacent
      subsets of $Y$.
\item $F$ has {\em strong continuity} if for each pair of
      adjacent $x,y \in X$, every point of $f(x)$ is adjacent
      or equal to some point of $f(y)$ and every point of 
      $f(y)$ is adjacent or equal to some point of $f(x)$.
     \qed
\end{itemize}
\end{definition}

\begin{prop}
\label{mild-and-weak}
\rm{\cite{BoxSta16}}
Let $F: X \multimap Y$ be a multivalued function between
digital images. Then $F$ is connectivity preserving if and
only if $F$ has weak continuity and for all $x \in X$,
$F(x)$ is connected. \qed
\end{prop}

\begin{exl}
\rm{\cite{BoxSta16}}
\label{pt-images-discon}
If $F: [0,1]_{\Z} \multimap [0,2]_{\Z}$ is defined by
$F(0)=\{0,2\}$, $F(1)=\{1\}$, then $F$ has both weak and
strong continuity. Thus a multivalued function between
digital images that has weak or strong continuity need not
have connected point-images. By Theorem~\ref{mildadj} and
Proposition~\ref{pt-images-connected} it
follows that neither having weak continuity nor having
strong continuity implies that a multivalued function is
connectivity preserving or continuous.
$\Box$
\end{exl}

\begin{exl}
\rm{\cite{BoxSta16}}
Let $F: [0,1]_{\Z} \multimap [0,2]_{\Z}$ be defined by
$F(0)=\{0,1\}$, $F(1)=\{2\}$. Then $F$ is continuous and
has weak continuity but
does not have strong continuity. $\Box$
\end{exl}

\begin{prop}
\rm{\cite{BoxSta16}}
Let $F: X \multimap Y$ be a multivalued function between
digital images. If $F$ has strong continuity and for
each $x \in X$, $F(x)$ is connected, then $F$ is
connectivity preserving. \qed
\end{prop}

The following shows that not requiring the images of
points to be connected yields topologically unsatisfying 
consequences for weak and strong continuity.

\begin{exl}
\rm{\cite{BoxSta16}}
Let $X$ and $Y$ be nonempty digital images. Let
the multivalued function $f: X \multimap Y$ be defined by
$f(x)=Y$ for all $x \in X$.
\begin{itemize}
\item $f$ has both weak and strong continuity.
\item $f$ is connectivity preserving if and only if $Y$ is
      connected. \qed
\end{itemize}
\end{exl}

As a specific example~\cite{BoxSta16} consider $X= \{0\} \subset \Z$ and $Y = \{0,2\}$, all with $c_1$ adjacency. Then the function $F:X \multimap Y$ with $F(0) = Y$ has both weak and strong continuity, even though it maps a connected image surjectively onto a disconnected image.

\subsection{Shy maps and their inverses}
\begin{definition}
\label{shy-def}
\cite{Boxer05}
Let $f: X \rightarrow Y$ be a
continuous surjection of digital images. We say $f$ is
{\em shy} if
\begin{itemize}
\item for each $y \in Y$, $f^{-1}(y)$ is connected, and
\item for every $y_0,y_1 \in Y$ such that $y_0$ and $y_1$ are
      adjacent, $f^{-1}(\{y_0,y_1\})$ is
      connected. \qed
\end{itemize}
\end{definition}

Shy maps induce surjections on fundamental groups
~\cite{Boxer05}.
Some relationships between shy maps $f$ and their inverses
$f^{-1}$ as multivalued functions were studied in
~\cite{Boxer14,BoxSta16,Boxer16}.
We have the following.

\begin{thm}
\label{shy-thm}
\rm{\cite{BoxSta16,Boxer16}}
Let $f: X \to Y$ be a 
continuous surjection between digital images.
Then the following are equivalent.
\begin{itemize}
\item f is a shy map.
\item For every connected $Y_0 \subset Y$ , $f^{-1}(Y_0)$
      is a connected subset of $X$.
\item $f^{-1}: Y \multimap X$ is a connectivity preserving multi-valued function.
\item $f^{-1}: Y \multimap X$ is a multi-valued function with weak continuity   
      such that for all $y \in Y$, $f^{-1}(y)$ is a connected subset of $X$. 
      \qed
\end{itemize}
\end{thm}

\subsection{Other tools}
Other terminology we use includes the following.
Given a digital image $(X,\kappa) \subset \Z^n$ and $x \in X$, the set of points adjacent to $x \in \Z^n$ and the
neighborhood of $x$ in $\Z^n$
are, respectively,
\[N_{\kappa}(x) = \{y \in \Z^n \, | \, y \mbox{ is }
    \kappa\mbox{-adjacent to }x\},\]
\[N_{\kappa}^*(x) = N_{\kappa}(x) \cup \{x\}.
\]

\section{Extensions of normal product adjacency}
In this section, we define extensions of the normal product adjacency, as follows.

\begin{definition}
\label{NP_u-def}
Let $u$ and $v$ be positive integers, $1 < u \leq v$. Let $\{(X_i,\kappa_i)\}_{i=1}^v$ be
digital images. Let $NP_u(\kappa_1, \ldots, \kappa_u)$ be the adjacency
defined on the Cartesian product $\Pi_{i=1}^v X_i$ as follows.
For $x_i,x_i' \in X_i$, $p=(x_1, \ldots, x_v)$ and $q=(x_1', \ldots, x_v')$ are
$NP_u(\kappa_1, \ldots, \kappa_u)$-adjacent if and only if
\begin{itemize}
\item For at least 1 and at most $u$ indices $i$, $x_i$ and $x_i'$ are
      $\kappa_i$-adjacent, and
\item for all other indices $i$, $x_i=x_i'$. \qed
\end{itemize}
\end{definition}

Throughout this paper, the reader should be careful to
note that some of our results for
$NP_u(\kappa_1, \ldots, \kappa_v)$ are stated for
all $u \in \{1,\ldots,v\}$ and others are stated
only for $u=1$ or $u=v$.

\begin{prop}
\label{NP_2}
$NP(\kappa, \lambda) = NP_2(\kappa,\lambda)$. I.e., given $x,x' \in (X,\kappa)$
and $y,y' \in (Y,\lambda)$, $p=(x,y)$ and $p'=(x',y')$ are 
$NP(\kappa, \lambda)$-adjacent in $X \times Y$ if and only if $p$ and $p'$
are $NP_2(\kappa, \lambda)$-adjacent.
\end{prop}

\begin{proof} This follows immediately from Definitions~\ref{NP-def}
and~\ref{NP_u-def}.
\end{proof}

\begin{thm}
\label{2prod}
\rm{~\cite{BoxKar12}} For
$X \in \Z^m$, $Y \in Z^n$, $NP_2(c_m,c_n)=c_{m+n}$,
i.e., the normal product adjacency for
$(X,c_m) \times (Y,c_n)$ coincides with the $c_{m+n}$-adjacency for $X \times Y$. \qed
\end{thm}

Examples are also given in~\cite{BoxKar12} that show that if
$X \in \Z^m$, $Y \in Z^n$, and $a< m$ or $b < n$,
then $NP_2(c_a,c_b) \neq c_{a+b}$.

The following shows that $NP_v$ obeys a recursive property.

\begin{prop}
\label{inductive-equiv} Let $v>2$. Then
\[ NP_v(\kappa_1, \ldots, \kappa_v) = NP_2(NP_{v-1}(\kappa_1, \ldots, \kappa_{v-1}), \kappa_v). \]
\end{prop}

\begin{proof} Let $x_i, x_i' \in X_i$ for $1 \leq i \leq v$.
Then $p=(x_1, \ldots, x_v)$ and $p'= (x_1', \ldots, x_v')$ are
$NP_v(\kappa_1, \ldots, \kappa_v)$-adjacent if and only if
for at least 1 and at most $v$ indices $i$, $x_i$ and $x_i'$ are
$\kappa_i$-adjacent and for all other indices $i$, $x_i=x_i'$. Hence
$p=(x_1, \ldots, x_v)$ and $p'= (x_1', \ldots, x_v')$ are
$NP_u(\kappa_1, \ldots, \kappa_u)$-adjacent if and only if either
\begin{itemize}
\item $x_i$ and $x_i'$ are $\kappa_i$-adjacent for from 1 to $u$ indices among
      $\{1, \ldots, v-1\}$, $x_i=x_i'$ for all other indices among
      $\{1, \ldots, v-1\}$, and $x_v=x_v'$; or
\item $x_i$ and $x_i'$ are $\kappa_i$-adjacent for from 1 to $u-1$ indices among
      $\{1, \ldots, v-1\}$, $x_i=x_i'$ for all other indices among
      $\{1, \ldots, v-1\}$, and $x_v$ and $x_v'$ are $\kappa_v$-adjacent.
\end{itemize}
Thus, $p=(x_1, \ldots, x_v)$ and $p'= (x_1', \ldots, x_v')$ are
$NP_u(\kappa_1, \ldots, \kappa_v)$-adjacent if and only if $p$ and $p'$ are
$NP_2(NP_{u-1}(\kappa_1, \ldots, \kappa_{v-1}), \kappa_v)$-adjacent.
\end{proof}

Notice Proposition~\ref{inductive-equiv} may fail
to extend to $NP_u(\kappa_1, \ldots, \kappa_v)$ if $u < v$, as shown in the following (suggested by an example
in~\cite{BoxKar12}).

\begin{exl}
Let $x_i,x_i' \in (X_i,\kappa_i)$,
$i \in \{1,2,3\}$. Suppose
$x_1$ and $x_1'$ are $\kappa_1$-adjacent,
$x_2$ and $x_2'$ are $\kappa_2$-adjacent,
and $x_3=x_3'$. Then
$(x_1,x_2,x_3)$ and $(x_1',x_2',x_3')$ 
are $NP_2(\kappa_1,\kappa_2,\kappa_3)$-adjacent in $X_1 \times X_2 \times X_3$, but
$(x_1,x_2)$ and $(x_1',x_2')$ are not
$NP_1(\kappa_1,\kappa_2)$-adjacent in
$X_1 \times X_2$. Thus,
\[ NP_2(\kappa_1,\kappa_2,\kappa_3) \neq
   NP_2(NP_1(\kappa_1,\kappa_2),\kappa_3).
\qed \]
\end{exl}

\begin{thm}
Let $f,g: (X,\kappa) \to (Y,\lambda)$ be
functions. 
Let $H: X \times [0,m]_{\Z} \to Y$ be
a function such that $H(x,0)=f(x)$ and
$H(x,m)=g(x)$ for all $x \in X$. Then
$H$ is a homotopy if and only if
$H$ is $(NP_1(\kappa,c_1),\lambda)$-continuous.
\end{thm}

\begin{proof} In the following, we
consider arbitrary $(NP_1(\kappa,c_1),\lambda)$-adjacent $(x,t)$ and $(x',t')$
in $X \times [0,m]_{\Z}$ with $x,x' \in X$ and $t,t' \in [0,m]_{\Z}$. Such
points offer the following cases.
\begin{enumerate}
\item $x$ and $x'$ are  $\kappa$-adjacent and $t=t'$; or
\item $x=x'$ and $t$ and $t'$ are
      $c_1$-adjacent, i.e.,
      $|t-t'|=1$.
\end{enumerate}

Let $H$ be a homotopy.
Then $f$ and $g$ are continuous, and
given $(NP_1(\kappa,c_1),\lambda)$-adjacent $(x,t)$ and $(x',t')$
in $X \times [0,m]_{\Z}$, we consider the cases listed above.
\begin{itemize}
\item In case~1, since $H$ is a 
      homotopy, $H(x,t)$ and $H(x',t)=H(x',t')$ are equal or $\lambda$-adjacent.
\item In case~2, since $H$ is a 
      homotopy, $H(x,t)$ and $H(x',t')=H(x,t')$ are equal or $\lambda$-adjacent.
\end{itemize}
Therefore, $H$ is $(NP_1(\kappa,c_1),\lambda)$-continuous.

Suppose $H$ is $(NP_1(\kappa,c_1),\lambda)$-continuous. Then for
$\kappa$-adjacent $x,x'$,
$f(x)=H(x,0)$ and $H(x',0)=f(x')$
are equal or $\lambda$-adjacent, so
$f$ is continuous. Similarly, $g(x)=H(x,m)$ and $g(x')=H(x',m)$ 
are equal or $\lambda$-adjacent, so
$g$ is
continuous. Also, the continuity of $H$ implies that $H(x,t)$ and
$H(x',t)$ must be equal or $\lambda$-adjacent, so the induced function $H_t$ is $(\kappa,\lambda)$-continuous. 
For $c_1$-adjacent $t,t'$, the continuity of $H$ implies that
$H(x,t)$ and $H(x,t')$ are equal or
$\lambda$-adjacent, so the induced
function $H_x$ is continuous. By
Definition~\ref{htpy-2nd-def}, $H$ is
a homotopy.
\end{proof}

\section{$NP_u$ and maps on products}
Given functions $f_i: (X_i, \kappa_i) \to (Y_i, \lambda_i)$,
$1 < i \leq v$, the function
\[ \Pi_{i=1}^v f_i: (\Pi_{i=1}^v X_i, NP_u(\kappa_1, \ldots, \kappa_v)) \to (\Pi_{i=1}^v Y_i, NP_u(\lambda_1, \ldots, \lambda_v))\]
is defined by
\[ \Pi_{i=1}^v f_i(x_1, \ldots, x_v) = (f(x_1), \ldots, f(x_v)), \mbox{ where } x_i \in X_i. \]

The following generalizes a result
in~\cite{Boxer16,BoxKar12}.
\begin{thm}
\label{prod-cont}
Let $f_i: (X_i, \kappa_i) \to (Y_i, \lambda_i)$, $1 < i \leq v$.
Then the product map
\[ f=\Pi_{i=1}^v f_i: (\Pi_{i=1}^v X_i, NP_v(\kappa_1, \ldots, \kappa_v)) \to (\Pi_{i=1}^v Y_i, NP_v(\lambda_1, \ldots, \lambda_v)) \]
is continuous if and only if each $f_i$ is continuous.
\end{thm}

\begin{proof} In the following, we
let $p=(x_1, \ldots, x_v)$ and $p'=(x_1', \ldots, x_v')$, where
$x_i, x_i' \in X_i$.

Suppose each $f_i$ is continuous and $p$ and $p'$ are $NP_v(\kappa_1, \ldots, \kappa_v)$-adjacent.
Then for all indices $i$, $x_i$ and $x_i'$ are equal or
$\kappa_i$-adjacent, so
$f_i(x_i)$ and $f_i(x_i')$ are equal or $\lambda_i$-adjacent.
Therefore, $f(p)$ and $f(p')$ are equal or
$NP_v(\lambda_1, \ldots, \lambda_v)$-adjacent. Thus, $f$ is continuous.

Suppose $f$ is continuous and for all indices~$i$, $x_i$ and $x_i'$
are $\kappa_i$-adjacent. Then $f(p)$ and $f(p')$ are equal or
$NP_v(\lambda_1, \ldots, \lambda_i)$-adjacent. Therefore, for each index $i$,
$f_i(x_i)$ and $f_i(x_i')$ are equal or $\lambda_i$-adjacent. Thus, each $f_i$ is
continuous.
\end{proof}

The statement analogous to 
Theorem~\ref{prod-cont} is not generally true if $c_u$-adjacencies
are used instead of normal product adjacencies, as shown in the following.

\begin{exl}
\label{factors-not-prod}
Let $X = \{(0,0), (1,0)\} \subset \Z^2$.
Let $Y = \{(0,0), (1,1)\} \subset \Z^2$.
Clearly, there is an isomorphism $f: (X,c_2) \to (Y,c_2)$. Consider
$X' = X \times \{0\} \subset \Z^3$ and
$Y' = Y \times \{0\} \subset \Z^3$. Note that the product map
$f \times 1_{\{0\}}$ is not
$(c_1,c_1)$-continuous, since
$X'$ is $c_1$-connected and
$Y'=(f \times 1_{\{0\}})(X')$ is not $c_1$-connected. \qed
\end{exl}

The following is a generalization of a result of~\cite{Han05}.

\begin{thm}
\label{projection-cont}
The projection maps
$p_i: (\Pi_{i=1}^v X_i, NP_u(\kappa_1, \ldots, \kappa_v)) \to (X_i, \kappa_i)$
defined by $p_i(x_1, \ldots, x_v) = x_i$
for $x_i \in (X_i,\kappa_i)$, are all continuous, for $1 \leq u \leq v$.
\end{thm}

\begin{proof} Let $p=(x_1, \ldots, x_v)$ and
$p'=(x_1', \ldots,x_v')$ be
$NP_u(\kappa_1, \ldots, \kappa_v)$-adjacent in
$(\Pi_{i=1}^v X_i, NP_u(\kappa_1, \ldots, \kappa_v))$, where
$x_i,x_i' \in X_i$. Then for all indices $i$,
$x_i=p_i(p)$ and $x_i'=p_i(p')$ are equal or $\kappa_i$-adjacent. Thus, $p_i$ is continuous.
\end{proof}

The statement analogous to 
Theorem~\ref{projection-cont} is not generally true if a $c_u$-adjacency
is used instead of a normal product adjacency, as shown in the following.

\begin{exl}
\cite{BoxKar12}
Let $X=[0,1]_{\Z} \subset \Z$.
Let $Y=\{(0,0),(1,1)\} \subset \Z^2$.
Then the projection map
$p_2: (X \times Y, c_3) \to (Y,c_1)$ is
not continuous, since $X \times Y$ 
is $c_3$-connected and $Y$ is not
$c_1$-connected. \qed
\end{exl}

We see in the next result that isomorphism is preserved by taking
Cartesian products with a normal product
adjacency.

\begin{thm}
\label{prod-iso}
Let $X = \Pi_{i=1}^v X_i$.
Let $f_i: (X_i,\kappa_i) \to (Y_i\,\lambda_i)$, $1 \leq i \leq v$.
\begin{itemize}
\item For $1 \leq u \leq v$, if
the product map
$f=\Pi_{i=1}^v f_i: (X, NP_u(\kappa_1,\ldots,\kappa_v)) \to 
(\Pi_{i=1}^v Y_i, NP_u(\lambda_1,\ldots,\kappa_v))$ is an isomorphism, then
for $1 \leq i \leq v$, $f_i$ is an
isomorphism.
\item If $f_i$ is an isomorphism for all $i$, then the product map
$f=\Pi_{i=1}^v f_i: (X, NP_v(\kappa_1,\ldots,\kappa_v)) \to 
(\Pi_{i=1}^v Y_i, NP_v(\lambda_1,\ldots,\kappa_v))$ is an isomorphism.
\end{itemize}
\end{thm}

\begin{proof}
Let $f$ be an isomorphism. Then each $f_i$ must be
one-to-one and onto.

Let $x_i \in X_i$. Let
$I_i: X_i \to X$ be defined by
\[I_i(x)=(x_1, \ldots, x_{i-1},x, x_{i+1},\ldots,x_v).
\]
Define $I_i': Y_i \to Y$ similarly.
Clearly, $I_i$ is $(\kappa_i,NP_u(\kappa_1,\ldots,\kappa_v)$-continuous
and $I_i'$ is $(\lambda_i,NP_u(\lambda_1,\ldots,\lambda_v)$-continuous.
Let $p_i: X \to X_i$ and $p_i': Y \to Y_i$ be the projections to the
$i$-th coordinate.
By Theorems~\ref{composition} and~\ref{projection-cont},
$f_i=p_i' \circ f \circ I_i$ and
$f_i^{-1}= p_i \circ f^{-1} \circ I_i'$
are continuous. Hence, $f_i$ is an
isomorphism.

Let $f_i: X_i \to Y_i$
be an isomorphism. One sees
easily that $f$ is one-to-one and onto,
and by Theorem~\ref{prod-cont}, $f$ is
continuous. The inverse function $f^{-1}$
is the product function of the $f_i^{-1}$, hence
is continuous by Theorem~\ref{prod-cont}.
Thus, $f$ is an isomorphism.
\end{proof}

The statement analogous to Theorem~\ref{prod-iso} is not generally
true for all $c_u$-adjacencies, as shown
by the following.

\begin{exl}
\label{prod-map-exl}
Let $X = \{(0,0), (1,1)\} \subset \Z^2$.
Let $Y = \{(0,0), (1,0)\} \subset \Z^2$.
Clearly, $(X,c_2)$ and $(Y,c_2)$ are
isomorphic. Consider
$X' = X \times \{0\} \subset \Z^3$ and
$Y' = Y \times \{0\} \subset \Z^3$. Note
$(X',c_1)$ and $(Y',c_1)$ are not
isomorphic, since the former is $c_1$-disconnected and the latter is
$c_1$-connected. \qed
\end{exl}

\section{$NP_v$ and connectedness}
\begin{thm}
\label{prod-connected}
Let $(X_i,\kappa_i)$ be digital images, $i \in \{1,2, \ldots, v\}$. Then
$(X_i,\kappa_i)$ is connected for all $i$ if and only
$(\Pi_{i=1}^v X_i, NP_v(\kappa_1, \ldots, \kappa_v))$ is connected.
\end{thm}

\begin{proof}
Suppose $(X_i,\kappa_i)$ is connected for all $i$. Let $x_i, x_i' \in X_i$.
Then there are paths $P_i$ in $X_i$ from $x_i$ to $x_i'$.
Let $p=(x_1,\ldots, x_v), p'=(x_1', \ldots, x_v') \in \Pi_{i=1}^v X_i$. Then
$\bigcup_{i=1}^v P_i'$, where
\[ P_1' = P_1 \times \{(x_2,\ldots,x_v)\}, \]
\[ P_i' = \{(x_1', \ldots, x_{i-1}')\} \times P_i \times \{(x_{i+1}, \ldots, x_v)\}
  \mbox{ for } 2 \leq i < v,
\]
\[ P_v' = \{(x_1', \ldots, x_{v-1}')\} \times P_v,
\]
is a path in $\Pi_{i=1}^v X_i$ from $p$ to $p'$. Since $p$ and $p'$ are 
arbitrarily selected points in $\Pi_{i=1}^v X_i$, it follows that
$(\Pi_{i=1}^v X_i, NP_v(\kappa_1, \ldots, \kappa_v))$ is connected.

If $(\Pi_{i=1}^v X_i, NP_v(\kappa_1, \ldots, \kappa_v))$ is connected, then
$(X_i,\kappa_i) = p_i(\Pi_{i=1}^v X_i)$ is connected,
by Definition~\ref{continuous} and 
Theorem~\ref{projection-cont}.
\end{proof}

The statement analogous to Theorem~\ref{prod-connected}
is not generally true if a $c_u$-adjacency is
used instead of $NP_v(\kappa_1, \ldots, \kappa_v)$
for $X \times Y$, as shown by the following.

\begin{exl}
\rm{\cite{BoxKar12}} Let $X=[0,1]_{\Z}$,
$Y=\{(0,0),(1,1)\} \subset \Z^2$. Then
$X \times Y$ is $c_2$-connected, but $Y$ is not
$c_1$-connected. Also, $X$ is $c_1$-connected
and $Y$ is $c_2$-connected, but $X \times Y$ is
not $c_1$-connected. \qed
\end{exl}

\section{$NP_v$ and homotopy relations}
In this section, we show that normal products preserve a variety of
digital homotopy relations. These include homotopy type and several generalizations introduced in~\cite{BoSt1}.
These generalizations -
homotopic similarity, long homotopy type, and real homotopy type -
all coincide with homotopy type on pairs of finite digital images; however, for
each of these relationships, an example
is given in~\cite{BoSt1} of a pair of digital images, at least one member of
which is infinite, such that the two
images have the given relation but are not
homotopy equivalent.

By contrast with Euclidean topology, in which a bounded space such as a single point and an unbounded
space such as $\R^n$ with Euclidean topology can have the same homotopy type, a finite digital
image and an image with infinite diameter - e.g., a single point
and $(\Z^n, c_1)$ - cannot share the same
homotopy type. However, examples in~\cite{BoSt1}
show that a finite digital
image and an image with infinite diameter
can share homotopic similarity, long homotopy type, or real homotopy type.

\subsection{Homotopic maps and homotopy type}
\begin{thm}
\label{prod-htpy}
Let $(X_i,\kappa_i)$ and $(Y_i,\lambda_i)$ be
digital images, $1 \leq i \leq v$. Let
$X = (\Pi_{i=1}^v X_i, NP_v(\kappa_1, \ldots, \kappa_i))$. Let
$Y = (\Pi_{i=1}^v Y_i, NP_v(\lambda_1, \ldots, \lambda_i))$.
Let $f_i,g_i: X_i \to Y_i$ be continuous and let
$H_i: X_i \times [0,m_i]_{\Z} \to Y_i$ be a homotopy
from $f_i$ to $g_i$. Then there is a homotopy $H$
between the product maps
$F=\Pi_{i=1}^v f_i: X \to Y$ and
$G=\Pi_{i=1}^v g_i: X \to Y$. If the homotopies $H_i$
are pointed, then $H$ is pointed.
\end{thm}

\begin{proof}
Let $M = \max\{m_i\}_{i=1}^v$. Let
$H_i': X_i \times [0,M]_{\Z} \to Y_i$ be defined by
\[ H_i'(x,t) = \left \{ \begin{array}{ll}
           H_i(x,t) & \mbox{for } 0 \le t \leq m_i; \\
           H_i(x,m_i) & \mbox{for } m_i \le t \le M.
\end{array} \right .
\]
Clearly, $H_i'$ is a homotopy from $f_i$ to $g_i$.

Let $H: X \times M \to Y$ be defined by
\[ H((x_1,\ldots,x_v),t) = (H_1'(x_1,t), \ldots, H_v'(x_v,t)).
\]
It is easily seen that $H$ is a homotopy from $F$ to $G$, and that
if each $H_i$ is pointed, then $H$ is pointed.
\end{proof}

The following theorem generalizes results
of~\cite{BoSt1}.

\begin{thm}
\label{product-htpy-thm}
Suppose
\begin{equation}
\label{factor-hyp}
X_i \simeq_{\kappa_i,\lambda_i} Y_i
\mbox{ for } 1 \leq i \leq v.
\end{equation}
Then
\begin{equation}
\label{product-htpy}
X=\Pi_{i=1}^v X_i \simeq_{NP_v(\kappa_1,\ldots,\kappa_v),NP_v(\lambda_1,\ldots,\lambda_v)} Y=\Pi_{i=1}^v Y_i.
\end{equation}
Further, if the homotopy equivalences~(\ref{factor-hyp})
are all pointed 
with respect to $x_i \in X_i$ and $y_i \in Y_i$, then
the homotopy equivalence~(\ref{product-htpy}) is pointed
with respect to $(x_1,\ldots,x_v) \in X$ and
$(y_1,\ldots,y_v) \in Y$.
\end{thm}

\begin{proof}
We give a proof for the unpointed assertion. With
minor modifications, the pointed assertion is proven
similarly.

By hypothesis, there exist continuous functions
$f_i: X_i \to Y_i$ and $g_i: Y_i \to X_i$
and homotopies
$H_i: X_i \times [0,m_i]_{\Z} \to X_i$ from
$g_i \circ f_i$ to $1_{X_i}$ and
$K_i: Y_i \times [0,n_i]_{\Z} \to Y_i$ from
$f_i \circ g_i$ to $1_{Y_i}$.

Let $M = \max\{m_i\}_{i=1}^v$. Then
$H_i': X_i \times [0,M]_{\Z} \to X_i$, defined by 
\[ H_i'(x,t) = \left \{ \begin{array}{ll}
   H_i(x,t) & \mbox{for } 0 \leq t \leq m_i; \\
   H_i(x,m_i) & \mbox{for } m_i \leq t \leq M,
   \end{array} \right .
\]
is clearly a homotopy from $g_i \circ f_i$ to $1_{X_i}$.

Let $F=\Pi_{i=1}^v f_i: X \to Y$.
Let $G=\Pi_{i=1}^v g_i: Y \to X$.
By Theorem~\ref{prod-cont}, $F$ and $G$ are
continuous.
Let $H: X \times [0,M]_{\Z} \to X$ be defined
by
\[ H(x_1,\ldots,x_v,t) =
   (H_1'(x_1,t),\ldots, H_v'(x_v,t)).
\]
Then $H$ is easily seen to be a homotopy from $G \circ 
F$ to $\Pi_{i=1}^v 1_{X_i} = 1_X$.

We can similarly show that $F \circ G \simeq 1_Y$. 
Therefore, $X \simeq Y$.
\end{proof}

The statements analogous to Theorems~\ref{prod-htpy}
and~\ref{product-htpy-thm} are not generally true if
a $c_u$-adjacency is used instead of a normal
product adjacency for the Cartesian product. 
Consider, e.g., $X$ and $Y$ as in
Example~\ref{prod-map-exl}. Let
$f: Y \to X$ be a $(c_1,c_2)$-isomorphism. Then
$f$ is $(c_1,c_2)$-homotopic to the constant map $\overline{(0,0)}$ of $Y$ to $(0,0)$. However,
$f \times 1_{\{0\}}$ is not even
$(c_1,c_1)$-continuous, hence is not $(c_1,c_1)$-homotopic to 
$\overline{(0,0)} \times 1_{\{0\}}$. Although $X$
and $Y$ are $(c_2,c_1)$-homotopy equivalent,
$X' = X \times \{0\}$ and $Y'=Y \times \{0\}$
are not $(c_1,c_1)$-homotopy equivalent, since
$X'$ is not $c_1$-connected and $Y'$ is
$c_1$-connected.

\subsection{Homotopic similarity}
\begin{definition}
\label{htpy-sim-def}
\rm{\cite{BoSt1}}
Let $X$ and $Y$ be digital images.
We say $(X,\kappa)$ and $(Y,\lambda)$
are {\rm homotopically similar}, denoted
$X \simeq_{\kappa,\lambda}^s Y$, if
there exist subsets
$\{X_j\}_{j=1}^{\infty}$ of $X$ and
$\{Y_j\}_{j=1}^{\infty}$ of $Y$ such that:
\begin{itemize}
\item $X = \bigcup_{j=1}^{\infty} X_j$, 
$Y = \bigcup_{j=1}^{\infty} Y_j$, 
and, for all $j$,
$X_j \subset X_{j+1}$,
$Y_j \subset Y_{j+1}$. 
\item There are continuous functions
      $f_j: X_j \rightarrow Y_j$, 
      $g_j: Y_j \rightarrow X_j$ such that
      $g_j \circ f_j \simeq_{\kappa,\kappa} 1_{X_j}$ and
      $f_j \circ g_j \simeq_{\lambda,\lambda} 1_{Y_j}$.
\item For $m \leq n$, 
      $f_n|X_m \simeq_{\kappa,\lambda} f_m$ in $Y_m$ and
      $g_n|Y_m \simeq_{\lambda,\kappa} g_m$ in $X_m$.
\end{itemize}

If all of these homotopies are pointed with respect to
some $x_1 \in X_1$ and $y_1 \in Y_1$,
we say $(X,x_1)$ and $(Y,y_1)$ are 
{\rm pointed homotopically similar},
denoted $(X,x_1) \simeq_{\kappa,\lambda}^s (Y,y_1)$
or $(X,x_1) \simeq^s (Y,y_1)$ when $\kappa$ and
$\lambda$ are understood.
$\Box$
\end{definition}

\begin{thm}
\label{htpy-sim}
Let $X_i \simeq_{\kappa_i,\lambda_i}^s Y_i$, $1 \leq i \leq v$.
Let $X = \Pi_{i=1}^v X_i$, $X = \Pi_{i=1}^v X_i$. Then
$X \simeq_{NP_v(\kappa_1,\ldots,\kappa_v), NP_v(\lambda_1,\ldots,\lambda_v)}^s Y$. If the similarities $X_i \simeq_{\kappa_i,\lambda_i}^s Y_i$ are pointed at $x_i \in X_i$, $y_i \in Y_i$,
then the similarity
$X \simeq_{NP_v(\kappa_1,\ldots,\kappa_v), NP_v(\lambda_1,\ldots,\lambda_v)}^s Y$
is pointed at $x_0=(x_1,\ldots,x_v) \in X$,
$y_0=(y_1,\ldots,y_v) \in Y$.
\end{thm}

\begin{proof}
We give a proof for the unpointed assertion. A virtually 
identical argument can be given for the pointed assertion.

By hypothesis, for $j \in \N$ there exist digital images
$X_{i,j} \subset X_i$, $Y_{i,j} \subset Y_i$ such that
$X_{i,j} \subset X_{i,j+1}$, $X_i = \bigcup_{j=1}^{\infty} X_{i,j}$,
$Y_{i,j} \subset Y_{i,j+1}$, $Y_i = \bigcup_{j=1}^{\infty} Y_{i,j}$, and continuous functions
$f_{i,j}: X_{i,j} \to Y_{i,j}$, $g_{i,j}: Y_{i,j} \to X_{i,j}$,
such that
$g_{i,j} \circ f_{i,j} \simeq_{\kappa_i,\kappa_i} 1_{X_{i,j}}$,
$f_{i,j} \circ g_{i,j} \simeq_{\lambda_i,\lambda_i} 1_{Y_{i,j}}$, and $m \le n$ implies
$f_{i,n}|X_{i,m} \simeq_{\kappa_i,\lambda_i} f_{i,m}$ in $Y_{i,m}$ and
$g_{i,n}|X_{i,m} \simeq_{\lambda_i,\kappa_i} g_{i,m}$ in $X_{i,m}$.

Let $X_j = \Pi_{i=1}^v X_{i,j}$,
$Y_j = \Pi_{i=1}^v Y_{i,j}$. Clearly we have
$X = \bigcup_{j=1}^{\infty} X_j$,
$Y = \bigcup_{j=1}^{\infty} Y_j$,
$X_j \subset X_{j+1}$, $Y_j \subset Y_{j+1}$. Let
$f_j = \Pi_{i=1}^v f_{i,j}: X_j \to Y_j$,
$g_j = \Pi_{i=1}^v g_{i,j}: Y_j \to X_j$.
By Theorem~\ref{prod-cont},
$f_j$ is $(NP_v(\kappa_1,\dots,\kappa_v), NP_v(\lambda_1,\ldots,\lambda_v))$-continuous and $g_j$ is
$(NP_v(\lambda_1,\ldots,\lambda_v), NP_v(\kappa_1,\dots,\kappa_v))$-continuous.
By Theorem~\ref{prod-htpy},
$g_j \circ f_j \simeq_{NP_v(\kappa_1,\dots,\kappa_v), NP_v(\kappa_1,\dots,\kappa_v)} 1_{X_j}$
and $f_j \circ g_j \simeq_{NP_v(\lambda_1,\ldots,\lambda_v), NP_v(\lambda_1,\ldots,\lambda_v)} 1_{Y_j}$.
Also by Theorem~\ref{prod-htpy}, $m \leq n$ implies
$f_n|X_m \simeq_{NP_v(\kappa_1,\dots,\kappa_v),NP_v(\lambda_1,\ldots,\lambda_v)} f_m$ in $Y_m$ and
      $g_n|Y_m \simeq_{NP_v(\lambda_1,\ldots,\lambda_v),NP_v(\kappa_1,\dots,\kappa_v)} g_m$ in $X_m$.
This completes the proof.
\end{proof}

\subsection{Long homotopy type}
\begin{definition}
\label{revised-L}
\rm{\cite{BoSt1}}
Let $(X,\kappa)$ and $(Y,\lambda)$ be digital
images. Let $f,g: X \rightarrow Y$ be continuous.
Let $F: X \times \Z \to Y$ be
a function such that
\begin{itemize}
\item for all $x\in X$, there exists 
$N_{F,x} \in \N$ such that $t \leq -N_{F,x}$ implies
$F(x,t)=f(x)$ and $t \geq N_{F,x}$ implies
$F(x,t)=g(x)$.
\item For all $x\in X$, the induced function $F_x:\Z \to Y$ defined by 
\[ F_x(t) = F(x,t) \text{ for all } t\in \Z \]
is $(c_1,\lambda)$-continuous. 
\item For all $t\in \Z$, the induced function $F_t:X \to Y$ defined by
\[ F_t(x) = F(x,t) \text{ for all } x\in X \]
is $(\kappa,\lambda)$-continuous.
\end{itemize}
Then $F$ is a {\rm long homotopy} from $f$ to $g$.
If for some $x_0 \in X$ and $y_0 \in Y$ we have
$F(x_0,t)=y_0$ for all $t \in \N^*$, we say
$F$ is a {\rm pointed long homotopy}. We write
$f \simeq_{\kappa,\lambda}^L g$, or 
$f \simeq^L g$ when the adjacencies $\kappa$ and
$\lambda$ are understood, to indicate that
$f$ and $g$ are long homotopic functions.
$\Box$
\end{definition}

We have the following.
\begin{thm}
\label{long-htpc}
Let $f_i,g_i: (X_i,\kappa_i) \to (Y_i,\lambda_i)$ be
continuous functions that are long homotopic, $1 \leq i \leq v$.
Then $f=\Pi_{i=1}^v f_i$ and $g=\Pi_{i=1}^v g_i$ are long
homotopic maps from $(\Pi_{i=1}^v X_i, NP_v(\kappa_1,\ldots,\kappa_v))$ to $(\Pi_{i=1}^v Y_i, NP_v(\lambda_1,\ldots,\lambda_v))$. If the long homotopies $f_i \simeq^L g_i$ are
pointed with respect to $x_i \in X_i$ and $y_i \in Y_i$, then
the long homotopy $f \simeq^L g$ is pointed with respect to
$(x_1, \ldots, x_v) \in \Pi_{i=1}^v X_i$ and
$(y_1, \ldots, y_v) \in \Pi_{i=1}^v Y_i$.
\end{thm}

\begin{proof}
We give a proof for the unpointed assertion. Minor modifications
yield a proof for the pointed assertion.

Let $h_i: X_i \times \Z \to Y_i$ be a long homotopy from
$f_i$ to $g_i$. For all $x_i \in X_i$, there exists
$N_{F_i,x_i} \in \N$ such that $t \le -N_{F_i,x_i}$ implies
$h_i(x_i,t)=f_i(x_i)$ and $t \ge N_{F_i,x_i}$ implies
$h_i(x_i,t)=g_i(x_i)$. For all
$x=(x_1,\ldots, x_v) \in \Pi_{i=1}^v X_i$, let
$N_x=\max\{N_{F_i,x_i} \, | \, 1 \le i \le v\}$. Let
$h=\Pi_{i=1}^v h_i: \Pi_{i=1}^v X_i \times \Z \to \Pi_{i=1}^v  Y_i$. Then $t \le -N_x$ implies $h(x,t)=f(x)$ and
$t \ge N_x$ implies $h(x,t)=g(x)$.

For all $x \in \Pi_{i=1}^v X_i$, the induced function
$h_x(t)=(h_i(x_1,t), \ldots, h_v(x_v,t))$ is $(c_1,NP_v(\lambda_1,\ldots,\lambda_v))$-continuous, by an
argument similar to that given in the proof of
Theorem~\ref{prod-htpy}.

For all $t \in \Z$, the induced function
$h_t(x)=(h_i(x_1,t), \ldots, h_v(x_v,t))$ is
$(NP_v(\kappa_1,\ldots,\kappa_v), NP_v(\lambda_1,\ldots,\lambda_v))$-continuous, by an
argument similar to that given in the proof of
Theorem~\ref{prod-htpy}. The assertion follows.
\end{proof}

\begin{definition}
\label{long-equiv-def}
\rm{\cite{BoSt1}}
Let $f: (X, \kappa) \rightarrow (Y,\lambda)$ and
$g: (Y, \lambda) \rightarrow (X,\kappa)$ be
continuous functions. Suppose
$g \circ f \simeq^L 1_X$ and
$f \circ g \simeq^L 1_Y$. Then we say
$(X,\kappa)$ and $(Y,\lambda)$  
have the {\em same long homotopy type}, denoted 
$X\simeq_{\kappa,\lambda}^L Y$ or simply $X \simeq^L Y$.
If there exist $x_0 \in X$ and $y_0 \in Y$ such
that $f(x_0)=y_0$, $g(y_0)=x_0$,
the long homotopy
$g \circ f \simeq^L 1_X$
holds $x_0$ fixed, and
the long homotopy
$f \circ g \simeq^L 1_Y$
holds $y_0$ fixed, then 
$(X,x_0,\kappa)$ and $(Y,y_0,\lambda)$ have the 
same {\em pointed long homotopy type}, denoted
$(X,x_0) \simeq_{\kappa,\lambda}^L (Y,y_0)$ or
$(X,x_0) \simeq^L (Y,y_0)$.
$\Box$
\end{definition}

\begin{thm}
\label{long-htpy-type}
Let $X_i \simeq_{\kappa_i,\lambda_i}^L Y_i$, $1 \le i \le v$.
Let $X = \Pi_{i=1}^v X_i$, $Y=\Pi_{i=1}^v Y_i$. Then
$X \simeq_{NP_v(\kappa_1,\ldots,\kappa_v),NP_v(\lambda_1,\ldots,\lambda_v)}^L Y$.
If for each $i$ the long homotopy equivalence
$X_i \simeq_{\kappa_i,\lambda_i}^L Y_i$ is pointed with respect
to $x_i \in X_i$ and $y_i \in Y_i$, then the long homotopy equivalence
$X \simeq_{NP_v(\kappa_1,\ldots,\kappa_v),NP_v(\lambda_1,\ldots,\lambda_v)}^L Y$
is pointed with respect to $x_0=(x_1,\ldots, x_v) \in X$ and
$y_0=(y_1,\ldots, y_v) \in Y$.
\end{thm}

\begin{proof} This follows easily from
Definition~\ref{long-equiv-def} and
Theorem~\ref{long-htpc}.
\end{proof}

\subsection{Real homotopy type}
\begin{definition}
\rm{\cite{BoSt1}}
\label{real-path}
Let $(X,\kappa)$ be a digital image, and $[0,1] \subset \R$ be the unit interval. A function $f:[0,1] \to X$ is a \emph{real [digital] [$\kappa$-]path in $X$} if:
\begin{itemize}
\item there exists $\epsilon_0>0$ such that $f$ is constant on $(0,\epsilon_0)$ with constant value equal or  $\kappa$-adjacent to $f(0)$, and
\item for each $t\in (0,1)$ there exists $\epsilon_t >0$ such that $f$ is constant on each of the intervals $(t-\epsilon_t,t)$ and $(t,t+\epsilon_t)$, and these two constant values are equal or $\kappa$-adjacent, with at least one of them equal to $f(t)$, and
\item there exists $\epsilon_1>0$ such that $f$ is constant on $(1-\epsilon_1, 1)$ with constant value equal or  $\kappa$-adjacent to $f(1)$.
\end{itemize}
If $t=0$ and $f(0) \neq f((0,\epsilon_0))$, or $0<t<1$ and
the two constant 
values $f((t-\epsilon_t,t))$ and $f((t,t+\epsilon_t))$
are not equal, or $t=1$ and $f(1) \neq f((1-\epsilon_1,1))$, we say $t$ is a \emph{jump} of $f$.
\end{definition}

\begin{prop}
\rm{\cite{BoSt1}}
\label{finite-jumps}
Let $p, q \in (X, \kappa)$. 
Let $f : [a, b] \rightarrow X$ be a real $\kappa$-path
from $p$ to $q$. Then the number of jumps of $f$ is 
finite.
\end{prop}

\begin{definition}
\rm{\cite{BoSt1}}
\label{real-def}
Let $(X,\kappa)$ and $(Y,\kappa')$ be digital images, and let $f,g: X \to Y$ be $(\kappa,\kappa')$ continuous. Then a \emph{real [digital] homotopy} of $f$ and $g$ is a function $F:X \times [0,1] \to Y$ such that:
\begin{itemize}
\item for all $x\in X$, $F(x,0) = f(x)$ and $F(x,1) = g(x)$
\item for all $x \in X$, the induced function $F_x:[0,1]\to Y$ defined by
\[ F_x(t) = F(x,t)\, \text{for all $t \in [0,1]$} \]
is a real $\kappa$-path in $X$.
\item for all $t\in [0,1]$, the induced function $F_t:X\to Y$ defined by 
\[ F_t(x) = F(x,t) \, \text{for all $x\in X$} \]
is $(\kappa,\kappa')$--continuous.
\end{itemize}
If such a function exists we say $f$ and $g$ are
{\em real homotopic} and write $f\simeq^\R g$.
If there are points $x_0 \in X$ and $y_0 \in Y$ such
that $F(x_0,t)=y_0$
for all $t \in [0,1]$, we say $f$ and $g$ are
{\em pointed real homotopic}.
\end{definition}

\begin{definition}
\rm{\cite{BoSt1}}
\label{real-htpy-equiv}
We say digital images $(X,\kappa)$ and $(Y, \kappa')$
have the {\em same real homotopy type}, denoted
$X \simeq_{\kappa,\kappa'}^\R Y$ or $X \simeq^\R Y$ when
$\kappa$ and $\kappa'$ are understood,
if there are continuous functions $f: X \rightarrow Y$
and $g: Y \rightarrow X$ such that 
$g \circ f \simeq^\R 1_X$ and $f \circ g \simeq^\R 1_Y$.
If there exist $x_0 \in X$ and $y_0 \in Y$ such that
$f(x_0)=y_0$, $g(y_0)=x_0$, and the real homotopies
above are pointed with respect to $x_0$ and $y_0$, we say $X$ and $Y$ have the 
{\em same pointed real homotopy type}, denoted
$(X,x_0) \simeq_{\kappa,\kappa'}^\R (Y,y_0)$ or
$(X,x_0) \simeq^\R (Y,y_0)$.
\end{definition}

\begin{thm}
\label{real-prod-thm}
Suppose
\begin{equation}
\label{real-equiv-hyp}
X_i \simeq_{\kappa_i,\lambda_i}^{\R} Y_i
\mbox{ for } 1 \le i \le v.
\end{equation}
Let $X = \Pi_{i=1}^v X_i$, $Y = \Pi_{i=1}^v Y_i$. Then
\begin{equation}
\label{real-prod-equiv}
X \simeq_{NP_v(\kappa_1,\ldots,\kappa_v),NP_v(\lambda_1,\ldots,\lambda_v)}^{\R} Y.
\end{equation}
If the equivalences~\rm{(\ref{real-equiv-hyp})} are
all pointed with respect to $x_i \in X_i$ and
$y_i \in Y_i$, then the equivalence~\rm{(\ref{real-prod-equiv})} is pointed with respect to
$x_0=(x_1,\ldots,x_v) \in X$ and
$y_0=(y_1,\ldots,y_v) \in Y$.
\end{thm}

\begin{proof}
We give a proof for the unpointed assertion.
With minor modifications, the same argument
yields the pointed assertion.

By hypothesis, there exist continuous functions
$f_i: X_i \to Y_i$, $g_i: Y_i \to X_i$
and real homotopies
$h_i: X_i \times [0,1] \to X_i$ from $g_i \circ f_i$ to $1_{X_i}$,
$k_i: Y_i \times [0,1] \to X_i$ from $f_i \circ g_i$ to $1_{Y_i}$.

Let $f = \Pi_{i=1}^v f_i: X \to Y$. Let
$g = \Pi_{i=1}^v g_i: Y \to X$. 
For $x=(x_1,\ldots,x_v) \in X$ with
$x_i \in X_i$, let
$H: X \times [0,1] \to X$ be defined by
\[H(x,t) = (h_1(x_1,t),\ldots,h_v(x_v,t)).
\]
Then $H(x,0)=G \circ F(x)$ and $H(x,1)=x$.

For $x \in X$, the induced function
$H_x$ has jumps only at the finitely
many (by Proposition~\ref{finite-jumps}) jumps of the functions
$h_i$ by the $x_i$. It follows that
$H_x$ is a real $NP_v(\lambda_1,\ldots,\lambda_v)$-path in $Y$.

Let $x'=(x_1',\ldots,x_v')$ be
$NP_v(\kappa_1,\ldots,\kappa_v)$-adjacent to
$x$ in $X$. Then, for any $t \in [0,1]$,
$H_t(x')=(H_1(x_1',t),\ldots,H_v(x_v',t))$
is $NP_v(\lambda_1,\ldots,\lambda_v)$-adjacent to
$H_t(x)=(H_1(x_1,t),\ldots,H_v(x_v,t))$, since each $H_i$ is a real homotopy. Hence
$H_t$ is continuous.

Thus, $H$ is a real homotopy from
$G \circ F$ to $1_X$. A similar argument
lets us conclude that $F \circ G \simeq^{\R} 1_Y$. Therefore, $X \simeq^{\R} Y$.
\end{proof}

\section{$NP_v$ and retractions}
\begin{definition}
\rm{\cite{Borsuk,Boxer94}}
\label{retract-def}
Let $Y \subset (X,\kappa)$. A $(\kappa,\kappa)$-continuous
function $r: X \to Y$ is a {\em retraction}, and 
$A$ is a {\em retract of} $X$, if $r(y)=y$ for all
$y \in Y$. \qed
\end{definition}

\begin{thm}
\label{ret-thm}
Let $A_i \subset (X_i,\kappa_i)$, $i \in \{1, \ldots, v\}$. Then
$A_i$ is a retract of $X_i$ for all $i$ if and only if
$\Pi_{i=1}^v A_i$ is a retract of 
$(\Pi_{i=1}^v X_i, NP_v(\kappa_1, \ldots, \kappa_v))$.
\end{thm}

\begin{proof}
Suppose, for all $i$, $A_i$ is a retract of $X_i$. Let
$r_i: X_i \to A_i$ be a retraction. Then, by Theorem~\ref{prod-cont},
$\Pi_{i=1}^v r_i: \Pi_{i=1}^v X_i \to \Pi_{i=1}^v A_i$ is continuous, and
therefore is easily seen to be a retraction.

Suppose there is a retraction $r: \Pi_{i=1}^v X_i \to \Pi_{i=1}^v A_i$.
We construct retractions $r_j: X_j \to A_j$ as follows. Let $a_i \in A_i$.
Define $f_j: X_j \to \Pi_{i=1}^v X_i$ by
\[ f_j(x) = (a_1, \ldots, a_{j-1}, x, a_{j+1}, \ldots, a_v). \]
Clearly, $f_j$ is continuous. Then
$r_j = p_j \circ r \circ f_j$ is continuous, by Theorem~\ref{composition} and
Corollary~\ref{projection-cont}, and is easily seen to be a retraction.
\end{proof}

Let $A \subset (X,\kappa)$. We say $A$ is
a {\em deformation retract of} $X$ if there is
a $\kappa$-homotopy 
$H: X \times [0,m]_{\Z} \to X$ 
from $1_X$ to a retraction of $X$ to $A$.
If $H(a,t)=a$ for all
$(a,t) \in Y \times [0,m]_{\Z}$, we say $H$
is a {\em strong} deformation and
$A$ is a {\em strong deformation retract of}
$X$. We have the following.

\begin{thm}
\label{def-ret-thm}
Let $A_i \subset (X_i,\kappa_i)$, $i \in \{1, \ldots, v\}$. Then
$A_i$ is a (strong) deformation retract of $X_i$ for all $i$ if and only if
$A=\Pi_{i=1}^v A_i$ is a (strong) deformation retract of 
$X=(\Pi_{i=1}^u X_i, NP_v(\kappa_1, \ldots, \kappa_v))$.
\end{thm}

\begin{proof}
Suppose $A_i$ is a deformation retract
of $X_i$, $1 \leq i \leq v$. It follows from
Theorems~\ref{prod-htpy} and~\ref{ret-thm} and that $A$ is a 
deformation retract of $X$. If each $A_i$ is a strong
deformation retract of $X_i$, then by using the
argument in the proof of Theorem~\ref{prod-htpy} we
can construct a homotopy from $1_X$ to a retraction
of $X$ to $A$ that holds every point of $a$ fixed,
so $A$ is a strong deformation retract of $X$.

Suppose $A$ is a (strong) deformation retract of $X$.
This means there is a homotopy
$H: X \times [0,m]_{\Z} \to X$ from $1_X$ to a
retraction $r$ of $X$ onto $A$ (such that
$H(a,t)=a$ for all $(a,t) \in A \times [0,m]_{\Z}$).
Let $f_i: X_i \to X$ be as in Theorem~\ref{ret-thm}.
Let $H_i: X_i \times [0,m]_{\Z} \to X_i$ be
defined by
$H_i(x,t)=p_i(H(f_i(x),t))$. Then $H_i$ is a homotopy
between $p_i \circ f_i = 1_{X_i}$ and
$p_i \circ r \circ f_i$ (such that $H_i(a_i,t)=a_i$
for all $a_i \in A_i$). Since 
\[ p_i \circ r \circ f_i(X_j) \subset
p_i \circ r(X) = p_i(A) = A_i \]
and $a \in A_i$ implies
$p_i \circ r \circ f_i(a)=a$,
$p_i \circ r \circ f_i$ is a 
retraction. Thus, $A_i$ is a (strong)
deformation retract of $X_i$.
\end{proof}

\section{$NP_v$ and the digital Borsuk-Ulam theorem}
The Borsuk-Ulam Theorem of Euclidean topology 
states that if $f: S^n \to \R^n$ is a
continuous function, where $\R^n$ is
$n$-dimensional Euclidean space and
$S^n$ is the unit sphere in $\R^{n+1}$, i.e.,
\[ S^n = \{(x_1,\ldots, x_{n+1})\in \R^{n+1} \, | \,
            \sum_{i=1}^{n+1} x_i^2 = 1\},
\]
then there exists $x \in S^n$ such that
$f(-x)=f(x)$.  A ``layman's example" of this
theorem for $n=2$ is that there are
opposite points $x,-x$ on the earth's surface with
the same temperature and the same
barometric pressure.

We say a set $X \subset \Z^n$ is {\em symmetric
with respect to the origin} if for every
$x \in X$, $-x \in X$.

The assertion analogous to the Borsuk-Ulam
theorem is not generally true
in digital topology. An example is given
in~\cite{Boxer10} of a continuous function
$f: S \to \Z$ from a simple closed curve
$S \subset (\Z^2, c_2)$, a digital analog of
$S^1$, into the digital line $\Z$, with $S$
symmetric with respect to the origin, such that
$f(x) \neq f(-x)$ for all $x \in S$. However,
the papers~\cite{Boxer10,Staecker16} give
conditions under which a continuous
function $f$ from a digital version $S_n$ of
$S^n$ to $\Z^n$ must have a point $x \in S_n$
for which $f(x)$ and $f(-x)$ are equal or
adjacent.

In particular,~\cite{Staecker16} uses the 
``boundary" of a digital box
as a digital model of a Euclidean sphere.
Let $B_n = \Pi_{i=1}^n [-a_i,a_i]_{\Z}$, for
$a_i \in \N$. Let
\[ \delta B_n = \bigcup_{i=1}^n \{(x_1,\ldots,x_n) \in B_n \, | \, x_i \in \{-a_i,a_i\}\}.
\]

\begin{thm}
\label{BU-lit}
We have the following.
\begin{itemize}
\item \rm{\cite{Boxer10}} Let $(S,\kappa)$ be a 
      digital simple closed curve in $Z^n$ such
      that $S$ is symmetric with respect to the
      origin. Let $f: S \to \Z$ be a
      $(\kappa,c_1)$-continuous function. Then
      for some $x \in S$,
      $f(x)$ and $f(-x)$ are equal or
      $c_1$-adjacent, i.e., $|f(x)-f(-x)| \leq 1$.
\item \rm{\cite{Staecker16}}
      Let $u \in \{1,n-1\}$ and let$f: \delta B_n \to \Z^{n-1}$ be a
      $(c_n,c_u)$-continuous function. Then for
      some $x \in \delta B_n$, $f(x)$ and $f(-x)$ 
      are equal or $c_u$-adjacent. \qed
\end{itemize}
\end{thm}

Notice that for $X_i \subset \Z^{n_i}$, $\Pi_{i=1}^v X_i$ is symmetric with respect to
the origin of $\Z^{\sum_{i=1}^v n_i}$ if and 
only if $X_i$ is symmetric with respect to
the origin of $\Z^{n_i}$ for all
indices~$i$.

Suppose $m,n \in \N$, $1 \leq m \leq n$. Let's say a digital image
$S \subset \Z^{n+1}$ that is symmetric with
respect to the origin
{\em has the 
$(m, \kappa,\lambda)$-Borsuk-Ulam property} if for every
$(\kappa,\lambda)$-continuous function
$f: S \to \Z^m$ there exists $x \in X$ such
that $f(x)$ and $f(-x)$ are equal or
$\lambda$-adjacent in $\Z^n$.

We have the following.

\begin{thm}
Suppose
\begin{itemize}
\item $v>1$,
\item $S_i \subset \Z^{n_i+1}$ is symmetric
with respect to the origin of $\Z^{n_i+1}$
for $1 \leq i \leq v$, and
\item $\Pi_{i=1}^v S_i$ has the
$(m,NP_v(\kappa_1,\ldots,\kappa_v),
  NP_v(\lambda_1,\ldots,\lambda_v))$-Borsuk-Ulam property for 
some adjacencies $\kappa_i$ for $\Z^{n_i+1}$ and
$\lambda_i$ for $\Z^{n_i}$, where
$m=\sum_{i=1}^v n_i$.
\end{itemize}
Then, for all
$i$, $S_i$ has the $(n_i,\kappa_i,\lambda_i)$-Borsuk-Ulam property.
\end{thm}

\begin{proof}
Notice $\Pi_{i=1}^v S_i \subset \Z^{m+v}$.
Let $f: \Pi_{i=1}^v S_i \to \Z^m$ be
any function that is
$(NP_v(\kappa_1, \ldots, \kappa_v),
  NP_v(\lambda_1, \ldots, \lambda_v))$-continuous. By hypothesis,
there exists $x \in \Pi_{i=1}^v S_i$
such that $f(x)$ and $f(-x)$ are
equal or $NP_v(\lambda_1, \ldots, \lambda_v)$-adjacent.

In particular, we can let $f$ be
the product of arbitrary continuous functions
$f_i:S_i \to \Z^{n_i}$, since
if $f_i: S_i \to \Z^{n_i}$ is
$(\kappa_i,\lambda_i)$-continuous, then
by Theorem~\ref{prod-cont},
$f=\Pi_{i=1}^v f_i$ is
$(NP_v(\kappa_1,\ldots,\kappa_v),
  NP_v(\lambda_1,\ldots,\lambda_v))$-continuous. Therefore,
there exists $x=(x_1,\ldots,x_v) \in X$
where $x_i \in S_i$ such that
$f(x)=(f_1(x_1), \ldots, f_v(x_v))$ and $f(-x)=(f_1(-x_1), \ldots, f_v(-x_v))$ are equal or are
$NP_v(\lambda_1,\ldots,\lambda_v)$-adjacent. Hence, for all
indices~$i$, $f_i(x_i)$ and $f_i(-x_i)$
are equal or $\lambda_i$-adjacent.

Since the $f_i$ were arbitrarily chosen, the assertion follows.
\end{proof}

\section{$NP_u$ and the approximate fixed point property}
In both topology and digital topology,
\begin{itemize}
\item a {\em fixed point} of a continuous function
      $f: X \to X$ is a point $x \in X$ satisfying $f(x)=x$;
\item if every continuous $f: X \to X$ has a fixed point,
      then $X$ has the {\em fixed point property} (FPP).
\end{itemize}
However, a digital image $X$ has the FPP if and only
if $X$ has a single point~\cite{BEKLL}. Therefore, it
turns out that the {\em approximate fixed point property} is
more interesting for digital images.

\begin{definition}
\rm{\cite{BEKLL}}
\label{approxFP}
A digital image $(X,\kappa)$ has the
{\em approximate fixed point property (AFPP)} if every
continuous $f: X \to X$ has an {\em approximate fixed point},
i.e., a point $x \in X$ such that $f(x)$ is equal or
$\kappa$-adjacent to $x$. \qed
\end{definition}

A number of results concerning the AFPP were presented
in~\cite{BEKLL}, including the following.

\begin{thm}
\rm{\cite{BEKLL}}
\label{AFPP-iso}
Suppose $(X,\kappa)$ has the AFPP. Let $h : X \to Y$ be a
$(\kappa,\lambda)$-isomorphism. Then
$(Y,\lambda)$ has the AFPP. \qed
\end{thm}

\begin{thm}
\rm{\cite{BEKLL}}
\label{AFPP-ret}
Suppose $Y$ is a retract of $(X,\kappa)$. If $(X,\kappa)$
has the AFPP, then $(Y,\kappa)$ has the AFPP. \qed
\end{thm}

The following is a generalization of Theorem~5.10
of~\cite{BEKLL}.

\begin{thm}
Let $(X_i,\kappa_i)$ be digital images, $1 \leq i \leq v$.
Then for any $u \in \Z$ such that
$1 \leq u \leq v$, if
$(\Pi_{i=1}^v X_i, NP_u(\kappa_1, \ldots, \kappa_v))$
has the AFPP then
$(X_i,\kappa_i)$ has the AFPP for all $i$.
\end{thm}

\begin{proof}
Let $X=(\Pi_{i=1}^v X_i, NP_u(\kappa_1, \ldots, \kappa_v))$.

Suppose $X$
has the AFPP. Let $x_i \in X_i$. Let
\[ X_1' = X_1 \times \{(x_2, \ldots, x_v)\}, \]
\[   X_i' = \{(x_1,\dots,x_{i-1})\} \times X_i \times \{(x_{i+1}, \ldots, x_v)\} \mbox{ for } 2 \leq i < v, \]
\[ X_v' = \{(x_1,\ldots,x_{v-1})\} \times X_V.
\]
Clearly, each $X_i'$ is a retract of $X$ and is
isomorphic to $X_i$. By Theorems~\ref{AFPP-iso}
and~\ref{AFPP-ret}, $X_i$ has the AFPP.
\end{proof}

\section{$NP_v$ and fundamental groups}
Several versions of the fundamental
group for digital images exist in the literature, including those
of~\cite{Stout,Kong,Boxer99,BoSt0}. In this
paper, we use the version of~\cite{Boxer99},
which was shown in~\cite{BoSt0} to be
equivalent to the version developed in the
latter paper. Other papers cited in this
section use the version of the digital
fundamental group presented in~\cite{Boxer99}.

The author of~\cite{Han05} attempted to
study the fundamental group of a Cartesian
product of digital simple closed curves. Errors
of~\cite{Han05} were corrected
in~\cite{BoxKar12}.

The notion of a covering map~\cite{Han05} is 
often useful in computing the fundamental 
group. The following is a somewhat simpler
characterization of a covering map than
that given in~\cite{Han05}.

\begin{thm}
{\rm \cite{Boxer06}}
\label{cover-def-equiv}
Let $(E, \kappa)$ and $(B, \lambda)$ be 
digital images. Let $g: E \to B$ be a
$(\kappa,\lambda)$-continuous surjection. 
Then $g$ is a {\em $(\kappa,\lambda)$-covering map} if and only if for each $b \in B$, there is an index set $M$
such that
\begin{itemize}
\item $g^{-1}(N_{\lambda}^*(b,1,B))=\bigcup_{i \in M}N_{\kappa}^*(e_i,1,E)$ where $e_i \in g^{-1}(b)$;
\item if $i,j \in M$ and $i \neq j$ then $N_{\kappa}^*(e_i,1, E) \cap N_{\kappa}^*(e_j,1,E) = \emptyset$; and
\item the restriction map
$g|_{N_{\kappa}^*(e_i,1, E)}: N_{\kappa}^*(e_i,1, E) \to N_{\lambda}^*(b,1, B)$ is a $(\kappa,\lambda)$-isomorphism for all $i \in M$. \qed
\end{itemize}
\end{thm}

\begin{exl}
\rm{\cite{Han05}}
Let $C \subset \Z^n$ be a
simple closed $\kappa$-curve,
as realized by a $(c_1,\kappa)$-continuous surjection $f: [0,m-1]_{\Z} \to C$ such that $f(0)$ and $f(m-1)$ are $\kappa$-adjacent. Define
$g: \Z \to C$ by
$g(z)=f(z \mod m)$. Then
$g$ is a covering map. \qed
\end{exl}

\begin{prop}
\rm{\cite{BoxKar12}}
\label{cover2}
Suppose for $i \in \{1,2\}$, $g_i: E_i \to B_i$ is a
$(\kappa_i,\lambda_i)$-covering map. Then
$g_1 \times g_2: E_1 \times E_2 \to B_1 \times B_2$ is a
$(NP_2(\kappa_1,\kappa_2),NP_2(\lambda_1,\lambda_2))$-covering map.
\qed
\end{prop}

\begin{cor}
\label{prod-cover}
Suppose for $i \in \{1,\ldots,v\}$, $g_i: E_i \to B_i$ is a
$(\kappa_i,\lambda_i)$-covering map. Then
$\Pi_{i=1}^v g_i: \Pi_{i=1}^v E_i \to \Pi_{i=1}^v B_i$ is a
$(NP_v(\kappa_1,\ldots,\kappa_v), NP_v(\lambda_1,\ldots,\lambda_v))$-covering map.
\qed
\end{cor}

\begin{proof} This follows from Propositions~\ref{cover2}
and~\ref{inductive-equiv}, and Theorem~\ref{cover-def-equiv}.
\end{proof}

For the following theorem, it is useful to know 
that a digital simple closed curve $S$ is not contractible if and only if $|S|>4$ 
~\cite{Boxer99,Boxer10}.

\begin{thm}
{\rm \cite{KRR92,Boxer99,Han05}}
Let $S \subset (\Z^n,\kappa)$ be a digital simple closed 
$\kappa$-curve that is not contractible.
Let $s_0 \in S$. Then the fundamental
group of $S$ is $\Pi_1^{\kappa}(S,s_0) \approx \Z$.
\qed
\end{thm}

The following theorem was discussed in~\cite{Han05},
but the argument given for it in~\cite{Han05}
had errors. A correct proof was given
in~\cite{BoxKar12}.

\begin{thm}
\rm{\cite{BoxKar12}}
\label{fund-2prod}
Let $S_i \subset (\Z^{n_i},c_{n_i})$,
for $i \in \{1,2\}$, be a noncontractible digital
simple closed curve. Let $s_i \in S_i$. Then
the fundamental group
\[ \Pi_1^{c_{n_1+n_2}} (S_1 \times S_2, (s_1,s_2))
   \approx \Z^2. \qed
\]
\qed
\end{thm}

The significance of the adjacency $c_{n_1+n_2}$ in
the proof of Theorem~\ref{fund-2prod} is that,
per Theorem~\ref{2prod},
$NP(c_{n_1},c_{n_2}) = c_{n_1+n_2}$. Thus, trivial
modifications of the proof given
in~\cite{BoxKar12} for Theorem~\ref{fund-2prod}
yield the following generalization.

\begin{thm}
\label{fund-vprod}
For $i \in \{1,\ldots,v\}$, let 
$S_i \subset (\Z^{n_i},\kappa_i)$
be a noncontractible digital
simple closed curve. Let $s_i \in S_i$. Then
the fundamental group
\[ \Pi_1^{NP_v(\kappa_1,\ldots,\kappa_v)} (\Pi_{i=1}^v S_i , (s_1,\ldots,s_v))
   \approx \Z^v. \qed
\]
\end{thm}

Many results concerning digital covering maps depend 
on the radius 2 local isomorphism property (e.g., 
\cite{Han05',Boxer06,BoxKar08,BoxKar10,BoxKar12,Boxer10,BoxKar12b}).
We have the following.

\begin{definition}
\rm{\cite{Han05'}} Let $n \in \N$.
A $(\kappa,\lambda)$-covering
$(E,p,B)$ is a {\em radius~$n$ local
isomorphism} if, for all $i \in M$,
the restriction map
$p|_{N_{\kappa}^*(e_i,n)}: N_{\kappa}^*(e_i,n) \to N_{\lambda}^*(b_i,n)$ is an
isomorphism, where $e_i,b_i,M$ are
as in Theorem~\ref{cover-def-equiv}.
\end{definition}

\begin{lem}
\label{radn}
Let $x_i \in (X_i,\kappa_i)$. Then
\[ N_{NP_v(\kappa_1,\ldots,\kappa_v)}^*((x_1,\ldots,x_n),n) = \Pi_{i=1}^v N_{\kappa_i}^*(x_i,n). \]
\end{lem}

\begin{proof} Let $x=(x_1,\ldots, x_v)$.
Let $y \in N_{NP_v(\kappa_1,\ldots,\kappa_v)}^*(x,n)$. For some
$m \leq n$, there is a path
$\{y_i\}_{i=0}^m$ from $x$ to $y$.
Let $y_i=(y_{i,1}, \ldots, y_{i,v})$
where $y_{i,j} \in X_i$.
Since $y_i$ and $y_{i+1}$ are
$NP_v(\kappa_1,\ldots,\kappa_v)$-adjacent, $y_{i,j}$ and
$y_{i,j+1}$ are equal or $\kappa_i$-adjacent. Therefore,
$\{y_{i,j}\}_{j=1}^m$ is a 
$\kappa_i$ path in $X_i$ from
$y_{i,0}$ to $y_{i,m}$. Hence
$N_{NP_v(\kappa_1,\ldots,\kappa_v)}^*((x_1,\ldots,x_n),n) \subset \Pi_{i=1}^v N_{\kappa_1}^*(x_i,n)$.

Let $y=(y_1,\ldots,y_v) \in \Pi_{i=1}^v N_{\kappa_1}^*(x_i,n)$.
For each $i$ and for some $m_i \leq n$, there is
a $\kappa_i$-path $P_i=\{y_{i,j}\}_{j=1}^{m_i}$ from $x_i$ to $y_i$. There is no loss of generality in assuming $m_i=n$,
since we can take $P_i=\{y_{i,j}\}_{j=1}^n$ where $y_{i,j}=y_{i,m_i}$ for $m_i \leq j \leq n$. Then for each $i<n$,
$y_i'=(y_{i,1},\ldots,y_{i,v})$ 
and $y_{i+1}'=(y_{i+1,1},\ldots,y_{i+1,v})$ are
equal or $NP_v(\kappa_1,\ldots,\kappa_v)$-adjacent. Then
$\{y_i'\}_{i=1}^n$ is an
$NP_v(\kappa_1,\ldots,\kappa_v)$-path from $x$ to $y$. Thus, $\Pi_{i=1}^v N_{\kappa_1}(x_i,n) \subset N_{NP_v(\kappa_1,\ldots,\kappa_v)}(x,n)$.
The assertion follows.
\end{proof}

\begin{thm} For $1 \leq i \leq v$, let
$p_i: (E_i,\kappa_i) \to (B_i,\lambda_i)$
be continuous and let $n \in \N$.  If
$(E_i,p_i,B_i)$ is a covering and a
radius~$n$ local isomorphism for all~$i$,
then the product function
\[ \Pi_{i=1}^v p_i: \Pi_{i=1}^v E_i \to 
   \Pi_{i=1}^v B_i
\]
is a $(NP_v(\kappa_1,\ldots,\kappa_v),NP_v(\lambda_1,\ldots,\lambda_v)$ covering map that is a
radius~$n$ local isomorphism.
\end{thm}

\begin{proof} This follows from
Corollary~\ref{prod-cover} and
Lemma~\ref{radn}.
\end{proof}

\section{$NP_v$ and multivalued functions}
We study properties of multivalued functions
that are preserved by $NP_v$.

\subsection{Weak and strong continuity}
\begin{thm}
\label{weak-prod}
Let $F_i: (X_i, \kappa_i) \multimap (Y_i,\lambda_i)$ be multivalued functions for $1 \le i \le v$.
Let $X = \Pi_{i=1}^v X_i$, $Y = \Pi_{i=1}^v Y_i$, and $F = \Pi_{i=1}^v F_i: (X, NP_v(\kappa_1,\ldots,\kappa_v)) \multimap (Y, NP_v(\lambda_1,\ldots,\lambda_v))$. Then
$F$ has weak continuity if and only if each 
$F_i$ has weak continuity.
\end{thm}

\begin{proof}
Let $x_i$ and $x_i'$ be $\kappa_i$-adjacent in $X_i$. Then
$x=(x_1,\ldots,x_v)$ and
$x'=(x_1',\ldots,x_v')$ are
$NP_v(\kappa_1, \ldots, \kappa_v)$-adjacent in $X$.

The multivalued function $F$ has weak continuity $\Leftrightarrow$ 
for $x,x'$ as above, $F(x)$ and $F(x')$ are 
$NP_v(\lambda_1, \ldots, \lambda_v)$-adjacent subsets of $Y$
$\Leftrightarrow$
for each~$i$ and for all $x_i,x_i'$ as above, $F_i(x_i)$ and $F_i(x_i')$
are $\lambda_i$-adjacent subsets of $Y_i$
$\Leftrightarrow$ for each~$i$, $F_i$ has weak continuity.
\end{proof}

\begin{thm}
\label{strong-prod}
Let $F_i: (X_i, \kappa_i) \multimap (Y_i,\lambda_i)$ be multivalued functions for $1 \le i \le v$.
Let $X = \Pi_{i=1}^v X_i$, $Y = \Pi_{i=1}^v Y_i$, and $F = \Pi_{i=1}^v F_i: (X, NP_v(\kappa_1,\ldots,\kappa_v)) \multimap (Y, NP_v(\lambda_1,\ldots,\lambda_v))$. Then
$F$ has strong continuity if and only if each 
$F_i$ has strong continuity.
\end{thm}

\begin{proof}
Let $x_i$ and $x_i'$ be $\kappa_i$-adjacent in $X_i$. Then
$x=(x_1,\ldots,x_v)$ and
$x'=(x_1',\ldots,x_v')$ are
$NP_v(\kappa_1, \ldots, \kappa_v)$-adjacent in $X$.

The multivalued function $F$ has strong continuity $\Leftrightarrow$ 
for $x,x'$ as above, every point of $F(x)$ is $NP_v(\lambda_1, \ldots, \lambda_v)$-adjacent or equal to a point of $F(x')$ and every point of $F(x')$ is $NP_v(\lambda_1, \ldots, \lambda_v)$-adjacent or equal to a point of $F(x)$
$\Leftrightarrow$
for each~$i$ and for all $x_i,x_i'$ as above, every point of $F_i(x_i)$ is 
$\lambda_i$-adjacent or equal to a point of $F_i(x_i')$ and every point of $F_i(x_i')$ is $\lambda_i$-adjacent or equal to a point of $F_i(x_i)$
$\Leftrightarrow$ for each~$i$, $F_i$ has strong continuity.
\end{proof}

\subsection{Continuous multifunctions}
\begin{figure}
\includegraphics[height=2.75
in]{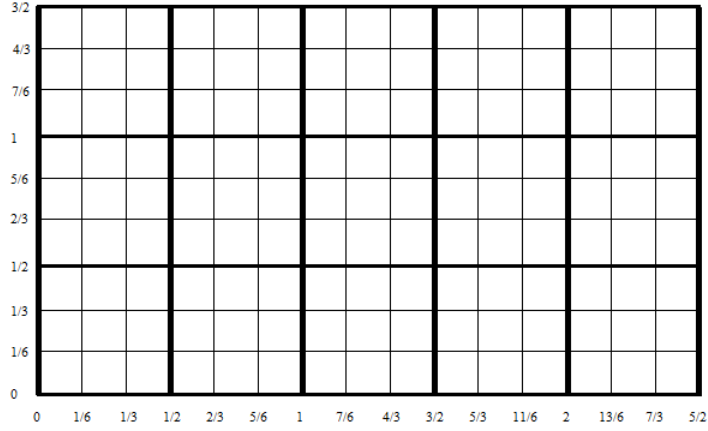}
\caption{The digital image
$X=[0,2]_{\Z}\times [0,1]_{\Z}$ with
its partitions $S(X,2)$ with member coordinates on heavy lines, and $S(X,6)$
with member coordinates on both heavy and light lines. In the notation used in the proof of
Lemma~\ref{gcm-subdiv}, we have,
e.g., $I(7/6,2/3)=(1,1/2)$.
}
\label{partitions-fig}
\end{figure}

\begin{lem}
\label{gcm-subdiv}
Let $X \subset \Z^m$, $Y \subset \Z^n$.
Let $F: (X,c_a) \multimap (Y,c_b)$ be a
continuous multivalued function. Let
$f: (S(X,r),c_a) \to (Y,c_b)$ be a continuous function that
induces $F$. Let $s \in \N$. Then there is a
continuous function
$f_s: (S(X,rs),c_a) \to (Y,c_b)$ that
induces $F$.
\end{lem}

\begin{proof} Given a point
$x=(x_1, \ldots, x_m) \in S(X,rs)$, there is
a unique point
$I(x)=x'=(x_1', \ldots, x_m') \in S(X,r)$ such that
$x'$ ``contains'' $x$ in the sense that the
fractional part of each component of $x$, 
$x_i - \lfloor x_i \rfloor$, ``truncates'' to the
fractional part of the corresponding component of $x'$, 
$x_i' - \lfloor x_i' \rfloor$, i.e.,
\[ x_i' - \lfloor x_i' \rfloor \leq x_i - \lfloor x_i \rfloor < x_i' - \lfloor x_i' \rfloor + 1/r.
\]
(See Figure~\ref{partitions-fig}.) Define $f_s(x)=f(I(x))$. 

We must show $f_s$ is a continuous multivalued function that induces $F$. If $x,x'$ are $c_a$-adjacent in
$S(X,rs)$, then one sees easily that
$I(x)$ and $I(x')$ are equal or
$c_a$-adjacent in $S(X,r)$. Hence
$f_s(x)=f(I(x))$ and $f_s(x')=f(I(x'))$ are equal or
$c_b$-adjacent in $Y$. Thus, $f_s$
is continuous. For $w \in X$ we have
\[ F(w) = \bigcup_{y \in E_r^{-1}(w)}f(y)
   = \bigcup_{u \in E_{rs}^{-1}(w)}f_s(u).
\]
Therefore, $f$ induces $F$.
\end{proof}

For multivalued functions
$F_i: (X_i,\kappa_i) \multimap (Y_i,\lambda_i)$,
$1 \leq i \leq v$,
define the product multivalued function
\[ \Pi_{i=1}^v F_i: (\Pi_{i=1}^v X_i, NP_v(\kappa_1,\ldots, \kappa_v)) \multimap (\Pi_{i=1}^v Y_i, NP_v(\lambda_1, \ldots, \lambda_v)) \]
by
\[ (\Pi_{i=1}^v F_i)(x_1,\ldots, x_v) =
   \Pi_{i=1}^v F_i(x_i).
\]

\begin{thm}
\label{multi-prod-thm}
Given multivalued functions
$F_i: (X_i,c_{a_i}) \multimap (Y_i,c_{b_i})$,
$1 \leq i \leq v$, if each
$F_i$ is continuous then the product multivalued function
\[ \Pi_{i=1}^v F_i: (\Pi_{i=1}^v X_i, NP_v(c_{a_1},\ldots, c_{a_v})) \multimap (\Pi_{i=1}^v Y_i, NP_v(c_{b_1}, \ldots, c_{b_v})) \]
is continuous.
\end{thm}

\begin{proof}
If each $F_i$ is continuous, there exists
a continuous
$f_i: (S(X_i,r_i),c_{a_i}) \to (Y_i,c_{b_i})$
that generates $F_i$. By Lemma~\ref{gcm-subdiv},
we may assume that all the $r_i$ are equal.
Thus, for some positive integer $r$, we have
$f_i: (S(X_i,r),c_{a_i}) \to (Y_i,c_{b_i})$
generating $F_i$.

By Theorem~\ref{prod-cont}, the product 
multivalued function
\[ \Pi_{i=1}^v f_i: (\Pi_{i=1}^v S(X_i,r),
NP_v(c_{a_1},\ldots,c_{a_v})) \to
(\Pi_{i=1}^v (Y_i,
NP_v(c_{b_1},\ldots,c_{b_v}))
\]
is continuous. It is clear that this function
generates the multivalued function
$\Pi_{i=1}^v F_i$.
\end{proof}

The paper~\cite{egs08} has several
results concerning the following notions.

\begin{definition}
\rm{\cite{egs08}}
Let $(X,\kappa) \subset \Z^n$ be a digital image
and $Y \subset X$. We say that $Y$ is
a {\em $\kappa$-retract of $X$} if there exists a $\kappa$-continuous multivalued
function $F: X \multimap Y$ (a 
{\em multivalued $\kappa$-retraction}) such that
$F(y) = \{y\}$ if $y \in Y$.
If moreover $F(x) \subset N_{c_n}^*(x)$ for every $x \in X \setminus Y$, we say
that $F$ is a {\em multivalued 
$(N, \kappa)$-retraction}, and
$Y$ is a {\em multivalued 
$(N, \kappa)$-retract} of $X$.
\end{definition}

We generalize Theorem~\ref{ret-thm} as follows.

\begin{thm}
\label{multi-ret-thm}
For $1 \leq i \leq v$, let $A_i \subset (X_i,\kappa_i) \subset \Z^{n_i}$.
Suppose $F_i: X_i \multimap A_i$ is a continuous multivalued function for all~$i$. Then $F_i$ is a multivalued retraction for all~$i$ if and only if
$F=\Pi_{i=1}^v F_i: \Pi_{i=1}^v X_i \multimap \Pi_{i=1}^v A_i$ is a multivalued
$NP_v(\kappa_1,\ldots,\kappa_v)$-retraction. Further,
$F_i$ is an $(N,\kappa_i)$-retraction for 
all~$i$ if and only if $F$ is a multivalued
$(N,NP_v(\kappa_1,\ldots,\kappa_v))$-retraction.
\end{thm}

\begin{proof}
Let $X=\Pi_{i=1}^v X_i$,
$A=\Pi_{i=1}^v A_i$.

Suppose each $F_i$ is a multivalued retraction.
By Theorem~\ref{multi-prod-thm},
the product multivalued function
$F$ is continuous.

Given $x=(x_1,\ldots,x_v) \in X \setminus A$, there exists
$j$ such that $x_j \in X_j \setminus A_j$, hence
$x \not \in F(A)$. Also, given
$a=(a_1,\ldots,a_v) \in A$, we have
\[ F(a) = \Pi_{i=1}^v F_i(a_i)=
   \Pi_{i=1}^v \{a_i\} = \{a\}.
\]
Therefore, $F(X)=A$, and $F$ is a multivalued
retraction.

Conversely, suppose $F$ is a multivalued
retraction. By Theorem~\ref{multi-prod-thm},
each $F_i$ is continuous.
Also, since $F(X)=A$, we must have, for each $i$,
$F_i(X_i)=A_i$, and since $F$ is a retraction, $F_i(a)=\{a\}$ for $a \in A_i$. Therefore, $F_i$ is a multivalued
retraction.

Further, from Lemma~\ref{radn},
for $x=(x_1,\ldots, x_v) \in X$,
$N_{NP_v(c_{n_1}, \ldots, c_{n_v})}^*(x)=
\Pi_{i=1}^v N_{c_{n_i}}(x_i)$. It
follows that $F_i$ is an $(N,\kappa_i)$-retraction for 
all~$i$ if and only if $F$ is a multivalued
$(N,NP_v(\kappa_1,\ldots,\kappa_v))$-retraction.
\end{proof}

\subsection{Connectivity preserving multifunctions}
\begin{thm}
Let $f_i: (X_i, \kappa_i) \multimap (Y_i, \lambda_i)$ be
a multivalued function between digital images,
$1 \leq i \leq v$. Then the product map
\[ \Pi_{i=1}^v f_i : (\Pi_{i=1}^v X_i, NP_v(\kappa_1, \ldots, \kappa_v)) \multimap (\Pi_{i=1}^v Y_i, NP_v(\lambda_1, \ldots, \lambda_v))
\]
is a connectivity preserving multifunction if and
only if each $f_i$
is a connectivity preserving multifunction.
\end{thm}

\begin{proof}
Let $X=\Pi_{i=1}^v X_i$, $Y=\Pi_{i=1}^v Y_i$, $F = \Pi_{i=1}^v f_i: X \multimap Y$. Assume 
\[ x=(x_1,\ldots,x_v),~x'=(x_1',\ldots,x_v')
\]
with $x_i,x_i' \in X_i$.
Using Theorem~\ref{mildadj}, we argue
as follows.

$F$ is connectivity preserving \\
$\Leftrightarrow$
\begin{itemize}
\item For every $x \in X$, $F(x)=\Pi_{i=1}^v F_i(x_i)$ is
      a connected subset of $Y$, and
\item For adjacent $x,x' \in X$,
      $F(x)=\Pi_{i=1}^v F_i(x_i)$ and $F(x')=\Pi_{i=1}^v F_i(x_i')$ are adjacent
      subsets of $Y$.
\end{itemize}
$\Leftrightarrow$
\begin{itemize}
\item For every $x_i \in X_i$, $F_i(x)$ is
      a connected subset of $Y_i$, and
\item For adjacent $x_i,x_i' \in X_i$,
      $F_i(x_i)$ and $F_i(x_i')$ are adjacent
      subsets of $Y_i$.
\end{itemize}
$\Leftrightarrow$ \\
each $F_i$ is
connectivity preserving.
\end{proof}

\section{$NP_v$ and shy maps}
The following generalizes a result
of~\cite{Boxer16}.

\begin{thm}
\label{shy-prod}
Let $f_i: (X_i, \kappa_i) \to (Y_i, \lambda_i)$ be
a continuous surjection between digital images,
$1 \leq i \leq v$. Then the product map
\[ f= \Pi_{i=1}^v f_i : (\Pi_{i=1}^v X_i, NP_v(\kappa_1, \ldots, \kappa_v)) \to (\Pi_{i=1}^v Y_i, NP_v(\lambda_1, \ldots, \lambda_v))
\]
is shy if and only if each $f_i$
is a shy map.
\end{thm}

\begin{proof}
Suppose the product map is shy.
Since $f_i = p_i \circ f$, it follows
from Theorems~\ref{composition} 
and~\ref{projection-cont} that $f_i$ is continuous. 
Also, since $f$ is surjective,
$f_i$ must be surjective.

Let $Y_i'$ be a $\lambda_i$-connected subset of $Y_i$.
By Theorem~\ref{prod-connected}, $\Pi_{i=1}^v Y_i'$ is
connected in $\Pi_{i=1}^v Y_i$.
Since the product map is shy, we have from
Theorem~\ref{shy-thm} that 
\[ X' = f^{-1}(\Pi_{i=1}^v Y_i') =
   \Pi_{i=1}^v f_i^{-1}(Y_i')
\]
is $NP_v(\kappa_1, \ldots, \kappa_v)$-connected. Then
$f_i^{-1}(Y_i')=p_i(X')$ is $\kappa_i$-connected.
From Theorem~\ref{shy-thm}, it follows that $f_i$ is
shy.

Conversely, suppose each $f_i$ is shy. By
Theorem~\ref{prod-cont}, the product map
$\Pi_{i=1}^v f_i$ is continuous, and it is easily
seen to be surjective.

Let $y_i \in Y_i$. Then
$(\Pi_{i=1}^v f_i)^{-1}(y_1,\ldots,y_v) =
 \Pi_{i=1}^v f_i^{-1}(y_i)$ is connected, by
 Definition~\ref{shy-def} and 
 Theorem~\ref{prod-connected}.

Let $y_i, y_i'$ be $\lambda_i$-adjacent in $Y_i$, and let
$y=(y_1,\ldots,y_v)$, $y'=(y_1',\ldots,y_v')$.
Then $y$ and $y'$ are adjacent in $Y$,
and
$(\Pi_{i=1}^v f_i)^{-1}(\{y,y'\}) =
 \Pi_{i=1}^v f_i^{-1}(\{y_i,y_i'\})$ is connected, by
 Definition~\ref{shy-def} and 
 Theorem~\ref{prod-connected}.

Thus, by Definition~\ref{shy-def}, $\Pi_{i=1}^v f_i$
is shy.
\end{proof}

The statement analogous to
Theorem~\ref{shy-prod} is not generally true if $c_u$-adjacencies
are used instead of normal product
adjacencies, as shown in the following.

\begin{exl}
Recall Example~\ref{factors-not-prod}, in which $X=\{(0,0),(1,0)\} \subset \Z^2$, $Y=\{(0,0),(1,1)\} \subset \Z^2$. There is a
$(c_1,c_2)$-isomorphism $f: X \to Y$. 
Consider $X' = X \times \{0\} \subset \Z^3$, $Y'= Y \times \{0\} \subset \Z^3$.
Although the maps $f$ and $1_{\{0\}}$ are,
respectively, $(c_1,c_2)$- and $(c_1,c_1)$-isomorphisms and therefore are, respectively, $(c_1,c_2)$- and $(c_1,c_1)$-shy,
the product map
$f \times 1_{\{0\}}: X' \to Y'$
is not $(c_1,c_1)$-shy, by
Theorem~\ref{shy-thm}, since, as
observed in Example~\ref{factors-not-prod}, $X'$ is $c_1$-connected
and $Y'$
is not $c_1$-connected. \qed
\end{exl}

\section{Further remarks}
We have studied adjacencies that are extensions of
the normal product adjacency for finite Cartesian
products of digital images. We have shown that such
adjacencies preserve many properties for finite 
Cartesian products of digital images that, in some
cases, are not preserved by the use of the 
$c_u$-adjacencies most commonly used in the 
literature of digital topology.

\section{Acknowledgment}
We are grateful for the remarks of P. Christopher
Staecker, who suggested this study and several of its
theorems, and helped with the proofreading.


\begin{thebibliography}{11}

\bibitem{Berge}
C. Berge,
{\em Graphs and Hypergraphs}, 2nd edition, North-Holland, Amsterdam, 1976.

\bibitem{Borsuk}
K. Borsuk,
{\em Theory of Retracts},
Polish Scientific Publishers, Warsaw, 1967.

\bibitem{Boxer94}
L. Boxer,
Digitally Continuous Functions,
{\em Pattern Recognition Letters} 15 (1994), 833-839.

\bibitem{Boxer99}
L. Boxer,
A Classical Construction for the Digital Fundamental Group,
{\em Pattern Recognition Letters} 10 (1999), 51-62.

\bibitem{Boxer05}
L. Boxer,
Properties of Digital Homotopy,
{\em Journal of Mathematical Imaging and Vision} 22 (2005),
19-26.

\bibitem{Boxer06}
L. Boxer, Digital Products, Wedges, and Covering Spaces,
{\em Journal of Mathematical Imaging and Vision} 25 (2006), 159-171. 

\bibitem{Boxer10}
L. Boxer, 
Continuous Maps on Digital Simple Closed Curves, 
{\em Applied Mathematics} 1 (2010), 377-386.

\bibitem{Boxer14}
L. Boxer,
Remarks on Digitally Continuous Multivalued Functions,
{\em Journal of Advances in Mathematics}
9 (1) (2014), 1755-1762.

\bibitem{Boxer16}
L. Boxer,
Digital Shy Maps, {\em Applied General Topology} 18 (1) 2017, 143-152.

\bibitem{BEKLL}
L. Boxer, O. Ege, I. Karaca, J. Lopez, and J. Louwsma, Digital Fixed Points, Approximate Fixed Points, and Universal Functions, 
{\em Applied General Topology}
17(2), 2016, 159-172. 

\bibitem{BoxKar08}
L. Boxer and I. Karaca,
The Classification of Digital Covering Spaces,
{\em Journal of Mathematical Imaging and Vision} 32 (1) (2008), 23-29.

\bibitem{BoxKar10}
L. Boxer and I. Karaca, Some Properties of Digital Covering Spaces,
{\em Journal of Mathematical Imaging and Vision} 37 (1) (2010), 17-26.

\bibitem{BoxKar12}
L. Boxer and I. Karaca,
Fundamental Groups for Digital Products,
{\em Advances and Applications in Mathematical Sciences} 11(4) (2012), 161-180.

\bibitem{BoxKar12b}
L. Boxer and I. Karaca, Actions of Automorphism Groups in a Digital Covering Space,
{\em Journal of Pure and Applied Mathematics: Advances and Applications} 8(1), 2012, 41-59.


\bibitem{BoxSta16}
L. Boxer and P.C. Staecker,
Connectivity Preserving Multivalued Functions in Digital Topology,
{\em Journal of Mathematical Imaging and Vision} 55 (3) (2016), 370-377.
DOI 10.1007/s10851-015-0625-5

\bibitem{BoSt0}
L. Boxer and P.C. Staecker, Remarks on Pointed Digital Homotopy, {\em Topology Proceedings} 51 (2018), 19-37. 

\bibitem{BoSt1}
L. Boxer and P.C. Staecker,
Homotopy relations for digital images, 
{\em Note di Matematica}, to appear.

\bibitem{Chen94}
L. Chen, Gradually varied surfaces and its optimal uniform approximation, {\em SPIE Proceedings}
2182 (1994), 300-307.

\bibitem{Chen04}
L. Chen, {\em Discrete Surfaces and Manifolds}, Scientific Practical Computing, Rockville, MD, 2004

\bibitem {egs08}
C. Escribano, A. Giraldo, and M. Sastre,
``Digitally Continuous Multivalued Functions,''
in \emph{Discrete Geometry for Computer Imagery}, Lecture Notes in Computer Science, v. 4992, Springer,
2008, 81--92.

\bibitem{egs12} 
C. Escribano, A. Giraldo, and M. Sastre,
``Digitally Continuous Multivalued Functions, Morphological Operations and Thinning Algorithms,''
\emph{Journal of Mathematical Imaging and Vision} 42 (2012), 76--91.

\bibitem{gs15}
A. Giraldo and M. Sastre,
On the Composition of Digitally Continuous Multivalued Functions,
{\em Journal of Mathematical Imaging and Vision} 53 (2) (2015), 196-209.

\bibitem{Han04}
S.-E. Han, An extended digital 
$(k_0, k_1)$-continuity,
{\em Journal of Applied Mathematics and Computing}, 16 (1-2) (2004), 445-452.

\bibitem{Han05'}
S.-E. Han,
Digital coverings and their applications,
{\em Journal of Applied Mathematics and Computing} 18 (2005), 487-495.

\bibitem{Han05}
S.-E. Han,
Non-product property of the digital fundamental group,
{\em Information Sciences} 171 (2005), 73-91.

\bibitem{Khalimsky}
E. Khalimsky,
Motion, deformation, and homotopy in finite spaces, in
{\em Proceedings IEEE International Conference on Systems, Man, and Cybernetics},
1987, 227-234.

\bibitem{Kong}
T.Y. Kong,
A digital fundamental group,
{\em Computers and Graphics}
13 (1989), 159-166.

\bibitem{KR}
T.Y. Kong and A. Rosenfeld, eds.
{\em Topological Algorithms for Digital
Image Processing}, Elsevier, 1996.

\bibitem{KRR92}
T.Y. Kong, A.W. Roscoe, and A. Rosenfeld,
Concepts of digital topology,
{\em Topology and its Applications} 46, 219-262.

\bibitem{Kovalevsky}
V.A. Kovalevsky,
A new concept for digital geometry,
{\em Shape in Picture},
Springer-Verlag, New York, 1994, pp. 37-51.

\bibitem{rosenfeld79}
A. Rosenfeld, 
Digital Topology, 
\emph{American Mathematical Monthly} 86 (1979), 621-630.

\bibitem{Rosenfeld}
A. Rosenfeld,
`Continuous' Functions on Digital Images,
{\em Pattern Recognition Letters} 4 (1987), 177-184.

\bibitem{Soille}
Soille, P.: Morphological operators. In: Jähne, B., et al. (eds.), Signal
Processing and Pattern Recognition. Handbook of Computer
Vision and Applications, vol. 2, pp. 627?682. Academic Press,
San Diego (1999).


\bibitem{Staecker16}
P.C. Staecker,
A Borsuk-Ulam theorem for digital images, submitted. Available at
http://arxiv.org/abs/1506.06426 

\bibitem{Stout}
Q.F. Stout,
Topological matching,
{\em Proceedings 15th Annual Symposium on Theory of Computing},
1983, 24-31.

\bibitem{Tsaur}
Tsaur, R., and Smyth, M.: ``Continuous" multifunctions in discrete spaces with applications to fixed point theory. In: Bertrand, G., Imiya, A., Klette, R. (eds.),
{\em Digital and Image Geometry}, 
Lecture Notes in Computer Science, vol. 2243, pp. 151-162. Springer Berlin / Heidelberg (2001), http://dx.doi.org/10.1007/3-540-45576-0 5, 10.1007/3-540-45576-0 5

\bibitem{vLW}
J.H. van Lint and R.M. Wilson,
{\em A Course in Combinatorics},
Cambridge University Press, New York, 1992

\end{thebibliography}
\end{document}